\documentclass[11pt]{article}

\usepackage[margin=1.25in]{geometry}

\usepackage{amssymb,amsmath}
\usepackage{amsthm}
\usepackage{hyperref}
\usepackage{mathrsfs}
\usepackage{textcomp}
\usepackage{enumerate}
\hypersetup{
    colorlinks,%
    citecolor=blue,%
    filecolor=blue,%
    linkcolor=blue,%
    urlcolor=black
}

\newtheorem{theorem}{Theorem}[section]

\newtheorem{lemma}[theorem]{Lemma}

\newtheorem{proposition}[theorem]{Proposition}
\newtheorem{corollary}[theorem]{Corollary}

\makeatletter

\newdimen\bibspace
\renewenvironment{thebibliography}[1]{%
 \section*{\refname %or \bibname if you use ``book'' as the documentclass
       \@mkboth{\MakeUppercase\refname}{\MakeUppercase\refname}}%
     \list{\@biblabel{\@arabic\c@enumiv}}%
          {\settowidth\labelwidth{\@biblabel{#1}}%
           \leftmargin\labelwidth
           \advance\leftmargin\labelsep
           \itemsep\bibspace
           \parsep\z@skip     %
           \@openbib@code
           \usecounter{enumiv}%
           \let\p@enumiv\@empty
           \renewcommand\theenumiv{\@arabic\c@enumiv}}%
     \sloppy\clubpenalty4000\widowpenalty4000%
     \sfcode`\.\@m}
    {\def\@noitemerr
      {\@latex@warning{Empty `thebibliography' environment}}%
     \endlist}

\makeatother

                \newcommand{\pa}{\partial}
           \newcommand{\ud}{\mathrm{d}}
\newcommand{\be}{\begin{equation}}      \newcommand{\ee}{\end{equation}}

\newcommand{\R}{\mathbb{R}}

\begin{document}

\title{Sharp Sobolev inequalities involving boundary terms revisited}

\author{{\medskip  Zhongwei Tang \thanks{Z. Tang is supported by NSFC (No. 12071036).},\ \  Jingang Xiong \footnote{J. Xiong is partially supported by NSFC (No. 11922104,11631002).}, \ \  Ning Zhou} }

\date{\today}

\maketitle

\begin{abstract}
We revisit  the sharp Sobolev  inequalities involving boundary terms on Riemannian manifolds with boundaries proved by \emph{[Y.Y. Li and M. Zhu, Geom. Funct. Anal. \textbf{8} (1998),  59--87.]} and explore the role of the mean curvature.
\end{abstract}

\section{Introduction}

Let $\R^n$ be the Euclidean space of dimension $n \geq 3$, and $2^*=2n/(n-2)$.  It was shown by Aubin \cite{AubinEspaces1976} and Talenti \cite{TalentiBest1976} that
\be\label{eq:aubin-talenti}
S^{-1}=\inf \Big\{\frac{\|\nabla u\|^2_{L^{2}(\mathbb{R}^{n})}}{\|u\|^2_{L^{2^{*}}(\mathbb{R}^{n})}}: u \in L^{2^{*}}(\mathbb{R}^{n}) \backslash\{0\},\, |\nabla u| \in L^{2}(\mathbb{R}^{n})\Big\}
\ee
is achieved and the extremal functions are found. Precisely,
$$
S=(\pi n(n-2))^{-1}\Big(\frac{\Gamma(n)}{\Gamma(n/2)}\Big)^{2/n},
$$
and the extremal functions must be the form
$$
v(x)=c\Big(\frac{\lambda}{\lambda^2+|x-x_0|^2}\Big)^{\frac{n-2}{2}}
$$
for some $c \in \R$, $\lambda>0,$ and $x_0 \in \R^{n}$. By the even reflection, we have, for all $u$ satisfying $u \in L^{2^*}(\mathbb{R}_{+}^{n})$ and $|\nabla u| \in L^{2}(\mathbb{R}_{+}^{n})$,
\be
\Big(\int_{\R^n_+} |u|^{2^*}\,\ud x\Big)^{2/2^*} \leq 2^{2 / n} S \int_{\R^n_+} |\nabla u|^2 \,\ud x ,
\ee
where $\R_{+}^{n}=\{x=(x_{1}, \cdots, x_{n}):x\in  \R^{n}, ~ x_{n}>0\}$ is the upper half space.

Let $(M, g)$ be a smooth compact $n$-dimensional Riemannian manifold with smooth boundary $\partial M$ and $n \geq 3$. If the manifold $(M, g)$ supports the inequality
\be\label{eq:mainine2}
\Big(\int_{M}|u|^{2^*}\, \ud v_{g} \Big)^{2/2^*}<2^{2 / n} S \int_{M}|\nabla_{g} u|^{2}\, \ud v_{g} \quad \mbox{for any } u \in H_{0}^{1}(M) \backslash\{0\},
\ee
 Li-Zhu \cite{LiZhuSharp1998} proved that there exists a constant $A(M,g)>0$ such that
\be \label{eq:LZ98}
\Big(\int_{M}|u|^{2^*}\, \ud v_{g} \Big)^{2/2^*}  \le 2^{2 / n} S   \int_{M}|\nabla_{g} u|^{2}\, \ud v_{g} +A(M,g) \int_{\pa M} u^2 \, \ud s_g
\ee
for any $u \in H^{1}(M) $, where $\ud v_{g}$ denotes the volume form of $g$ and $\ud s_g $ is the induced volume form on $\pa M$. Moreover, they proved that inequality \eqref{eq:mainine2} is necessary to \eqref{eq:LZ98}. When $M$ is a bounded smooth domain in $\mathbb{R}^n$ and $g$ is flat, the inequality \eqref{eq:mainine2} holds automatically and \eqref{eq:LZ98} was established by Br\'ezis-Lieb \cite{BrezisLiebSobolev1985} for $M$ being the unit ball and by Adimurthi-Yadava \cite{AdimurthiYadavaSome1994} for $n \geq 5$. Without the condition \eqref{eq:mainine2}, Li-Zhu \cite{LiZhuSharp1998} proved that
\be \label{eq:LZ98-b}
\Big(\int_{M}|u|^{2^*}\, \ud v_{g} \Big)^{2/2^*}  \le 2^{2 / n} S   \int_{M}\left|\nabla_{g} u\right|^{2}\, \ud v_{g} +A'(M,g)\Big(\int_{M} u^2\, \ud v_{g} +  \int_{\pa M} u^2 \, \ud s_g\Big),
\ee
where $A'(M,g)>0$ depends only on $M$ and $g$.

The inequalities \eqref{eq:LZ98} and \eqref{eq:LZ98-b} are sharp in the sense that the constant $2^{2 / n} S$ can not be replaced by a smaller constant. This is in the same spirit of a conjecture posed by Aubin \cite{AubinProblemes1976}, which has been confirmed through the work of Hebey-Vaugon \cite{HebeyVaugonMeilleures1996}, Aubin-Li \cite{AubinLiOn1999}, Druet \cite{DruetThe1999,DruetIsoperimetric2002}.

In this paper, we would like to examine the $L^2$ term of \eqref{eq:LZ98} and \eqref{eq:LZ98-b} and obtain some Sobolev inequalities involving the mean curvature. Our main theorems are as follows.

\begin{theorem}\label{thm:mainthm2}
Let $(M, g)$ be a smooth compact $n$-dimensional Riemannian manifold with smooth boundary $\partial M$ and $n \geq 7$.  Suppose that the manifold $(M, g)$ supports the inequality \eqref{eq:mainine2}.  Then there exists  a constant $A_1(M, g)>0$ such that
\be\label{eq:mainine1}
\Big(\int_{M}|u|^{2^*}\, \ud v_{g} \Big)^{2/2^*}    \leq 2^{2 / n} S\Big(\int_{M}|\nabla_{g} u|^{2}\, \ud v_{g}+\frac{n-2}{2} \int_{\partial M} h_{g} u^{2}\, \ud s_{g}\Big)+A_1\|u\|_{L^{r}(\partial M)}^{2}
\ee
for all  $u \in H^{1}(M)$, where $h_g$ is the mean curvature of $\pa M$ and $r=\frac{2(n-1)}{n}$.
\end{theorem}

\begin{theorem}\label{thm:mainthm3}
Let $(M, g)$ be a smooth compact $n$-dimensional Riemannian manifold with smooth boundary $\partial M$ and $n \geq 7$. Then there exists a constant $A_{2}(M, g)>0$ such that for all $u \in H^{1}(M)$,
\be\label{eq:mainine3}
\begin{aligned}
\Big(\int_{M}|u|^{2^*}\, \ud v_{g} \Big)^{2/2^*}    \leq& 2^{2 / n} S\Big(\int_{M}|\nabla_{g} u|^{2}\, \ud v_{g}+\frac{n-2}{2} \int_{\partial M} h_{g} u^{2}\, \ud s_{g}\Big)\\[2mm]& +A_2(\|u\|_{L^{r_1}(M)}^{2}+\|u\|_{L^{r_2}(\partial M)}^{2}),
\end{aligned}
\ee
where $r_1=\frac{2n}{n+2}$ and $r_2=\frac{2(n-1)}{n}$.
\end{theorem}

The inequality \eqref{eq:mainine1} and \eqref{eq:mainine3} are sharp. They would fail if $2^{2 / n} S$ is replaced by any smaller constant. In general, $h_g$ cannot be replaced by any smooth function which is smaller than $h_g$ at some point on $\partial M,$ and $r$, $r_1$, $r_2$ cannot be replaced by any smaller number.

The effect of the scalar curvature in the refined inequalities related to the Aubin conjecture was studied by Druet-Hebey-Vaugon \cite{DHV01}, Druet-Hebey \cite{DH02}, Li-Ricciardi \cite{LiRicciardiA2003} and etc. In particular, Li-Ricciardi \cite{LiRicciardiA2003} obtained  an inequality involving the scalar curvature in the way of the mean curvature in  \eqref{eq:mainine1}, if $n\ge 6$. For the Sobolev trace inequality, closely related results were obtained by Li-Zhu \cite{LiZhuSharp1997} and Jin-Xiong \cite{JinXiongSharp2013, JinXiongA2015}.

The paper is organized as follows. In section 2, we begin to prove Theorem \ref{thm:mainthm2} by a contradiction argument, we establish some preliminary results and use Moser iteration technique to obtain an appropriate upper bound for solutions to a
nonlinear critical exponent elliptic equation with Neumann boundary condition. In section 3, we adapt ideas of Bahri-Coron \cite{BahriCoronThe1991} to prove an energy estimate. In section 4, we finish the proof of Theorem \ref{thm:mainthm2}. Finally, in section 5, we prove Theorem \ref{thm:mainthm3}.

{\bf Notations:} We collect below a list of the main notations used throughout the paper.
\begin{itemize}
  \item For a domain $D \subset \mathbb{R}^{n}$ with boundary $\partial D,$ we denote $\partial^{\prime} D$ as the interior of $\overline{D} \cap \partial \mathbb{R}_{+}^{n}$ in $\partial \mathbb{R}_{+}^{n}=\mathbb{R}^{n-1}$ and $\partial^{\prime \prime} D=\partial D \backslash \partial^{\prime} D$.
  \item For $x_0 \in \mathbb{R}^{n}$, $B_{r}(x_0):=\{x \in \mathbb{R}^{n}:|x-x_0|<r\}$ and $B_{r}^{+}(x_0):=B_{r}(x_0) \cap \mathbb{R}_{+}^{n}$,
where $|x|=\sqrt{x_1^2+x_2^2+\cdots+x_n^2}$ for $x=(x_1,x_2,\cdots,x_n)$. We will not keep writing the center $x_0$ if $x_0=0$.
\end{itemize}

\medskip

\noindent \textbf{Acknowledgement:} J. Xiong would like to thank Tianling Jin for valuable discussions.

\section{Preliminaries}

We prove Theorem \ref{thm:mainthm2} by a contradiction argument. Suppose that for every large $\alpha>0$, there exists $\tilde{u} \in H^{1}(M)$ such that
$$
\Big(\int_{M}|\tilde{u}|^{2^*}\, \ud v_{g} \Big)^{2/2^*}>2^{2 / n} S\Big(\int_{M}|\nabla_{g} \tilde{u}|^{2} \, \ud v_{g}+\frac{n-2}{2} \int_{\partial M} h_{g} \tilde{u}^{2}\, \ud s_g\Big)+\alpha\|\tilde{u}\|_{L^{r}(\partial M)}^{2}.
$$
Define
$$
I_{\alpha}(u)=\frac{\int_{M}|\nabla_{g} u|^{2} \, \ud  v_{g}+\frac{n-2}{2} \int_{\partial M} h_{g} u^{2}\, \ud s_g+\alpha\|u\|_{L^{r}(\partial M)}^{2}}{(\int_{M}|u|^{2^*}\, \ud v_{g} )^{2/2^*}}, \quad \forall \, u \in H^{1}(M) \backslash \{0\}.
$$
It follows from the contradiction hypothesis that for all large $\alpha,$
\be\label{eq:ellalpha<}
\ell_{\alpha}:=\inf _{H^{1}(M) \backslash \{0\}} I_{\alpha}<\frac{1}{2^{2 / n} S}.
\ee

\begin{proposition}\label{pro:E-L}
For every $\alpha>0$, $\ell_{\alpha}$ is achieved by a nonnegative $u_{\alpha} \in H^{1}(M) \backslash \{0\}$ with
\be\label{eq:2*norm=1}
\int_{M} u_{\alpha}^{2^*} \, \ud  v_{g}=1.
\ee
Moreover, $u_{\alpha} \in C^{\infty}(M) \cap C^{1, r-1}(\overline{M})$ and satisfies the Euler-Lagrange equation
\be\label{eq:E-L}
\left\{\begin{array}{ll}
-\Delta_{g} u_{\alpha}=\ell_{\alpha} u_{\alpha}^{2^*-1} & \text { in } M, \\
\frac{\partial_{g} u_{\alpha}}{\partial \nu}=-\frac{n-2}{2} h_{g} u_{\alpha}-\alpha\|u_{\alpha}\|_{L^{r}(\partial M)}^{2-r} u_{\alpha}^{r-1} & \text { on } \partial M.
\end{array}\right.
\ee
\end{proposition}

\begin{proof}
The existence of a minimizer follows from the standard subcritical exponent approximating method and Moser's iteration argument; see, e.g., \cite[Proposition 2.1]{EscobarConformal1992}. A calculus of variational argument shows that $u_{\alpha}$ is a weak solution of the Euler-Lagrange equation \eqref{eq:E-L}. Finally, $u_{\alpha} \in C^{\infty}(M) \cap C^{1, r-1}(\overline{M})$ follows from the regularity theory in \cite{CherrierProblemes1984}.
\end{proof}

\begin{lemma}\label{lem:lem2.2}
For every $\varepsilon>0,$ there exists a constant $A(\varepsilon)>0$ depending on $\varepsilon$, $M$, and $g$ such that
$$
\Big(\int_{M}|u|^{2^*} \, \ud  v_{g}\Big)^{2 / 2^*} \leq(2^{2 / n} S+\varepsilon) \int_{M}|\nabla_{g} u|^{2} \, \ud  v_{g}+A(\varepsilon)\|u\|_{L^{r}(\partial M)}^{2}
$$
for all $u \in H^{1}(M)$.
\end{lemma}

\begin{proof}
By compactness, we have that for every $\varepsilon>0$ there exists a positive constant $\tilde{A}(\varepsilon)$ such that
\be\label{eq:Interpolation Inequality}
\int_{\partial M} u^{2} \, \ud  s_{g} \leq \varepsilon \int_{M}|\nabla_{g} u|^{2} \, \ud  v_{g}+\tilde{A}(\varepsilon)\|u\|_{L^{r}(\partial M)}^{2}.
\ee
Hence, the lemma follows from the inequality \eqref{eq:LZ98}.
\end{proof}

\begin{lemma}\label{lem:pro1.3}
Let $(M, g)$ be a smooth compact $n$-dimensional Riemannian manifold with smooth boundary and $n \geq 3$. Then for any $\varepsilon>0$, there exists some constant $B(\varepsilon)$ depending on $\varepsilon$ such that for all $u \in H^{1}(M)$,
\be
\Big(\int_{M}|u|^{2^*} \, \ud  v_{g}\Big)^{2 / 2^*}\leq(S+\varepsilon) \int_{M}|\nabla_{g} u|^{2} \, \ud  v_{g}+B(\varepsilon)(\|u\|_{2, M}^{2}+\|u\|_{L^{{2(n-1)}/{(n-2)}}(\partial M)}^{2}).
\ee
\end{lemma}

\begin{proof}
The proof can be found in Li-Zhu \cite[Proposition 1.3]{LiZhuSharp1998}.
\end{proof}

\begin{lemma}\label{lem:thm0.1}
Let $(M, g)$ be a smooth compact $n$-dimensional Riemannian manifold with smooth boundary and $n \geq 3$. Then there exists a constant $A^{\prime}(M, g)>0$ such that for all $u \in H^{1}(M)$,
\be\label{eq:lemthm0.1}
\Big(\int_{\partial M}|u|^{\frac{2(n-1)}{n-2}} \, \ud s_{g}\Big)^{\frac{n-2}{n-1}} \leq S_1 \int_{M}|\nabla_{g} u|^{2} \, \ud v_{g}+A^{\prime} \int_{\partial M} u^{2} \, \ud s_{g},
\ee
where $S_1=\frac{2}{n-2} \sigma_{n}^{-1 /(n-1)}$ with $\sigma_{n}$ the volume of the unit sphere in $\mathbb{R}^{n}$.
\end{lemma}

\begin{proof}
The proof can be found in Li-Zhu \cite[Theorem 0.1]{LiZhuSharp1997}.
\end{proof}

Using \eqref{eq:Interpolation Inequality}, we have
$$
\begin{aligned}
I_{\alpha}(u_{\alpha}) &=\int_{M}|\nabla_{g} u_{\alpha}|^{2} \, \ud  v_{g}+\frac{n-2}{2} \int_{\partial M} h_{g} u_{\alpha}^{2}\, \ud s_g+\alpha\|u_{\alpha}\|_{L^{r}(\partial M)}^{2} \\
& \geq(1-\varepsilon \max _{\partial M}|h_{g}|) \int_{M}|\nabla_{g} u_{\alpha}|^{2} \, \ud  v_{g}+(\alpha-{C}(\varepsilon))\|u_{\alpha}\|_{L^{r}(\partial M)}^{2},
\end{aligned}
$$
where $C(\varepsilon)$ is a positive constant depending on $\varepsilon$, $M$, and $g$. Then we derive by \eqref{eq:ellalpha<} that
$$
\int_{M}|\nabla_{g} u_{\alpha}|^{2} \, \ud  v_{g} \leq \frac{2}{2^{2/n}S}
$$
and
\be\label{eq:limrnorm0}
\left\|u_{\alpha}\right\|_{L^{r}(\partial M)} \rightarrow 0 \quad \text { as } \alpha \rightarrow \infty.
\ee
It follows that $u_{\alpha} \rightharpoonup \bar{u}$ in $H^{1}(M)$ for some $\bar{u} \in H_{0}^{1}(M)$.

We claim that, as $\alpha \rightarrow \infty$,
\be\label{eq:limellalpha}
\ell_{\alpha} \rightarrow \frac{1}{2^{2/n}S}
\ee
and
\be\label{eq:limitalphaualpha}
\alpha\left\|u_{\alpha}\right\|_{L^{r}(\partial M)}^{2} \rightarrow 0.
\ee
Indeed, by Lemma \ref{lem:lem2.2} and \eqref{eq:Interpolation Inequality}, for every $\varepsilon>0$,
$$
\begin{aligned}
1 & \leq(2^{2/n}S+\varepsilon/2) \int_{M}|\nabla_{g} u_{\alpha}|^{2} \, \ud  v_{g}+A(\varepsilon)\|u_{\alpha}\|_{L^{r}(\partial M)}^{2} \\
& \leq(2^{2/n}S+\varepsilon) \ell_{\alpha}+(2 A(\varepsilon)-\alpha 2^{2/n}S)\|u_{\alpha}\|_{L^{r}(\partial M)}^{2}.
\end{aligned}
$$
Thus
$$
\frac{1}{2^{2/n}S+\varepsilon} \leq \ell_{\alpha}<\frac{1}{2^{2/n}S}
$$
and
$$
\frac{1}{2} \alpha 2^{2/n}S\|u_{\alpha}\|_{L^{r}(\partial M)}^{2} \leq(2^{2/n}S+\varepsilon) \frac{1}{2^{2/n}S}-1=\frac{\varepsilon}{2^{2/n}S},
$$
if $\alpha 2^{2/n}S>4 A(\varepsilon) .$ Hence, the claim follows.

Let $x_{\alpha}$ be a maximum point of $u_{\alpha},$ and set $\mu_{\alpha}:=u_{\alpha}(x_{\alpha})^{-2 /(n-2)}$.

\begin{lemma}\label{lem:lemalphato0}
We have
\be\label{eq:limitalphamualpha}
\lim _{\alpha \rightarrow \infty} \alpha \mu_{\alpha}^{2}=0.
\ee
\end{lemma}

\begin{proof}
We first claim
\be\label{eq:claim}
\liminf _{\alpha \rightarrow \infty}\|u_{\alpha}\|_{L^{2(n-1)/(n-2)}(\partial M)}>0.
\ee
If the claim were false, i.e., $\|u_{\alpha}\|_{L^{{2(n-1)}/(n-2)}(\partial M)} \rightarrow 0$ along a subsequence $\alpha \rightarrow \infty$. Then, by \eqref{eq:ellalpha<} and \eqref{eq:2*norm=1}, there would exist $\hat{u} \in H_{0}^{1}(M)$ such that $u_{\alpha}$ weakly converges to $\hat{u}$. It follows from Br\'ezis-Lieb lemma that $u_{\alpha}$ and $\hat{u}$ satisfy
\be\label{eq:BrezisLieb1}
\int_{M} u_{\alpha}^{2^*}\, \ud v_g-\int_{M}|u_{\alpha}-\hat{u}|^{2^*}\, \ud v_g-\int_{M} \hat{u}^{2^*}\, \ud v_g \rightarrow 0 \quad \text { as } \alpha \rightarrow \infty,
\ee
and, in view of \eqref{eq:2*norm=1},
\be\label{eq:BrezisLieb2}
\int_{M}|u_{\alpha}-\hat{u}|^{2^*}\, \ud v_g \leq 1+o(1), \quad \int_{M} \hat{u}^{2^*}\, \ud v_g \leq 1,
\ee
where and throughout this paper $o(1) \rightarrow 0$ as $\alpha \rightarrow \infty$. By Lemma \ref{lem:pro1.3} applied to $u_{\alpha}-\hat{u},$ we have that for any $\varepsilon>0$,
\be\label{eq:lemmaapplied}
\int_{M}|\nabla_{g}(u_{\alpha}-\hat{u})|^{2}\, \ud v_g\geq(1 / S-\varepsilon)\Big(\int_{M}|u_{\alpha}-\hat{u}|^{2^*}\, \ud v_g\Big)^{2 / 2^*}.
\ee
By the definition of $\ell_{\alpha},$ we have
\be\label{eq:defiell}
\int_{M}|\nabla_{g} \hat{u}|^{2}\, \ud v_g\geq \ell_{\alpha}\Big(\int_{M} \hat{u}^{2^*}\, \ud v_g\Big)^{2/2^*}.
\ee
By the compact embedding of $H^{1}(M)$ to $L^{2}(M)$ and $L^{2}(\partial M)$, \eqref{eq:BrezisLieb1}, \eqref{eq:BrezisLieb2}, \eqref{eq:lemmaapplied} and \eqref{eq:defiell}, for any $\varepsilon>0$, we have
$$
\begin{aligned}
\ell_{\alpha} &=\int_{M}|\nabla_{g} u_{\alpha}|^{2}\, \ud v_g +\frac{n-2}{2} \int_{\partial M} h_{g} u_{\alpha}^{2}\, \ud s_g+\alpha\|u_{\alpha}\|_{L^{r}(\partial M)}^{2} \\
&=\int_{M}|\nabla_{g}(u_{\alpha}-\hat{u})|^{2}\, \ud v_g+\int_{M}|\nabla_{g} \hat{u}|^{2}\, \ud v_g+\alpha\|u_{\alpha}\|_{L^{r}(\partial M)}^{2}+o(1) \\
& \geq(1 / S-\varepsilon)\Big(\int_{M}|u_{\alpha}-\hat{u}|^{2^*}\, \ud v_g\Big)^{2 / 2^*}+\ell_{\alpha}\Big(\int_{M} \hat{u}^{2^*}\, \ud v_g\Big)^{2 / 2^*}+\alpha\|u_{\alpha}\|_{L^{r}(\partial M)}^{2}+o(1) \\
& \geq(1 / S-\varepsilon)\int_{M}|u_{\alpha}-\hat{u}|^{2^*}\, \ud v_g+\ell_{\alpha}\int_{M} \hat{u}^{2^*}\, \ud v_g+\alpha\|u_{\alpha}\|_{L^{r}(\partial M)}^{2}+o(1) \\
& =(1 / S-\varepsilon-\ell_{\alpha}) \int_{M}|u_{\alpha}-\hat{u}|^{2^*}\, \ud v_g+\ell_{\alpha}+\alpha\|u_{\alpha}\|_{L^{r}(\partial M)}^{2}+o(1).
\end{aligned}
$$
Taking $\varepsilon$ small enough and using \eqref{eq:ellalpha<}, we can derive that $\|u_{\alpha}-\hat{u}\|_{L^{2^*}(M)} \rightarrow 0 .$ In particular, in view of \eqref{eq:2*norm=1}, $\int_{M} \hat{u}^{2^*}\,\ud v_g=1 .$ This, together with \eqref{eq:ellalpha<}, implies that
$$
\Big(\int_{M}|\hat{u}|^{2^*} \, \ud  v_{g}\Big)^{2 / 2^*} \geq 2^{2 / n} S\int_{M}|\nabla_{g} \hat{u}|^{2} \, \ud  v_{g},
$$
which contradicts to \eqref{eq:mainine2}. This establishes \eqref{eq:claim}.

It follows from \eqref{eq:claim}, the definition of $\mu_{\alpha},$ and \eqref{eq:limitalphaualpha} that as $\alpha \rightarrow \infty$,
\be\label{eq:varepsilon1}
\alpha \mu_{\alpha}^{2} \leq C \alpha \mu_{\alpha}^{2} \Big(\int_{\partial M} u_{\alpha}^{\frac{2(n-1)}{n-2}} \, \ud  s_{g}\Big)^{2/r} \leq C \alpha \Big(\int_{\partial M} u_{\alpha}^{r}\, \ud s_g\Big)^{2/r} \rightarrow 0.
\ee
The proof of Lemma \ref{lem:lemalphato0} is completed.
\end{proof}

Let $(y^{1}, \cdots, y^{n-1}, y^{n})$ denote some geodesic normal coordinates given by the exponential map $\exp _{x_{\alpha}}$. In this coordinate system, the metric $g$ is given by $g_{i j}(y) \, \ud  y^{i} \, \ud  y^{j} .$

Let $\delta_{1}>0$ be the injectivity radius of $(M, g)$, we define $v_{\alpha}$ in a neighborhood of $z=0$ by
$$
v_{\alpha}(z)=\mu_{\alpha}^{(n-2) / 2} u_{\alpha}(\exp _{x_{\alpha}}(\mu_{\alpha} z)), \quad z \in O_{\alpha} \subset \mathbb{R}^{n},
$$
where
$$
O_{\alpha}=\{z \in \mathbb{R}^{n}:|z|<\delta_{1} / \mu_{\alpha},\, \exp _{x_{\alpha}}(\mu_{\alpha} z) \in M\}.
$$
We write $\partial O_{\alpha}=\Gamma_{\alpha}^{1} \cup \Gamma_{\alpha}^{2},$ where
$$
\Gamma_{\alpha}^{1}=\{z \in \partial O_{\alpha}: \exp _{x_{\alpha}}(\mu_{\alpha} z) \in \partial M\}, \quad \Gamma_{\alpha}^{2}=\{z \in \partial O_{\alpha}: \exp _{x_{\alpha}}(\mu_{\alpha} z) \in M\}.
$$
It follows from \eqref{eq:E-L} that $v_{\alpha}$ satisfies
$$
\left\{\begin{array}{ll}
-\Delta_{g_{\alpha}} v_{\alpha}=\ell_{\alpha} v_{\alpha}^{2^*-1} & \text { in } O_{\alpha}, \\
\frac{\partial_{g_{\alpha}} v_{\alpha}}{\partial \nu}=-\frac{n-2}{2}h_{\alpha}v_{\alpha}-\varepsilon_{\alpha}v_{\alpha}^{r-1} & \text { on } \Gamma_{\alpha}^{1}, \\
v_{\alpha}(0)=1,\quad  0 \leq v_{\alpha} \leq 1,
\end{array}\right.
$$
where $g_{\alpha}=g_{i j}(\mu_{\alpha} z) \, \ud z^{i} \, \ud z^{j}$, $h_{\alpha}$ is the mean curvature of $\Gamma_{\alpha}^{1}$ with respect to the metric $g_{\alpha}$ which satisfies $|h_{\alpha}|\leq C\mu_{\alpha}$, and
\be\label{eq:varepsilonalpha}
\varepsilon_{\alpha}:=\alpha \mu_{\alpha}^{n-1-\frac{n-2}{2} r}\|u_{\alpha}\|_{L^{r}(\partial M)}^{2-r}.
\ee
It follows from \eqref{eq:varepsilon1} that
\be\label{eq:varepsilon2}
\varepsilon_{\alpha}=\alpha\|u_{\alpha}\|_{L^{r}(\partial M)}^{2}\Big(\frac{\alpha\mu_{\alpha}^{2}}{\alpha(\int_{\partial M} u_{\alpha}^{r} \, \ud s_{g})^{2/r}}\Big)^{r/2} \leq C\alpha\|u_{\alpha}\|_{L^{r}(\partial M)}^{2} \rightarrow 0 \quad \text { as } \alpha \rightarrow \infty.
\ee
By the standard elliptic equations theory, for all $R>1$,
$$
\|v_{\alpha}\|_{C^{1,r-1}(B_{R} \cap \overline{O_{\alpha}})} \leq C(R) \quad \text { for all sufficiently large } \alpha.
$$
Up to a subsequence, let $\lim \limits_{\alpha \rightarrow \infty} \operatorname{dist}(x_{\alpha}, \partial M) / \mu_{\alpha}=T \in[0, \infty]$. Therefore $v_{\alpha} \rightarrow U$ in $C_{l o c}^{1}(\overline{\mathbb{R}_{-T}^{n}})$ for some $U \in C^{\infty}(\mathbb{R}_{-T}^{n}) \cap  C_{l o c}^{1, r-1}(\overline{\mathbb{R}_{-T}^{n}})$ which satisfies
\be\label{eq:U-equation}
\left\{\begin{array}{ll}
-\Delta U=\frac{1}{2^{2 / n} S} U^{2^*-1} & \text { in } \mathbb{R}_{-T}^{n}, \\
\frac{\partial U}{\partial \nu}=0 & \text { on } \partial \mathbb{R}_{-T}^{n}, \\
U(0)=1, \quad 0 \leq U(x) \leq 1,
\end{array}\right.
\ee
where $\mathbb{R}_{-T}^{n}=\{z=(z^{\prime}, z_{n}): z_{n}>-T\}$ for $0 \leq T<\infty ;$ for $T=\infty$, $\mathbb{R}_{-\infty}^{n}=\mathbb{R}^{n}$ and there is no boundary condition in the above.

\begin{lemma}\label{lem:lemT=0}
We have
$$
T=0.
$$
\end{lemma}

\begin{proof}
We prove it by a contradiction argument. If $0<T<\infty,$ then after making an even reflection of $U,$ we obtain a positive solution of $-\Delta U=2^{-2 / n} S^{-1} U^{2^*-1}$ in $\mathbb{R}^{n}$ with two local maximum points. This is impossible due to the results of Gidas-Ni-Nirenberg \cite{GidasNiNirenbergSymmetry1979}. If $T=\infty,$ we have
\be\label{eq:T1}
\int_{\mathbb{R}^{n}} U^{2^*}\, \ud x=2.
\ee
Indeed, we know from the explicit expression
of $U$ given by \cite{CaffarelliGidasSpruckAsymptotic1989} that $U$ is an extremal function for the Sobolev inequality, i.e.,
$$
\frac{\int_{\mathbb{R}^{n}}|\nabla U|^{2}\, \ud x}{(\int_{\mathbb{R}^{n}} U^{2^*}\, \ud x)^{2 / 2^*}}=\frac{1}{S}.
$$
Multiplying \eqref{eq:U-equation} by $U$ and integrating by parts, we obtain that
$$
\int_{\mathbb{R}^{n}}|\nabla U|^{2}\, \ud x=\frac{1}{2^{2 / n} S} \int_{\mathbb{R}^{n}}U^{2^*}\, \ud x.
$$
Equality \eqref{eq:T1} follows immediately from the above two identities. However for all $R>0,$
$$
\int_{B_{R}(0)} v_{\alpha}^{2^*} \, \ud  v_{g_{\alpha}}=\int_{B_{\mu_{\alpha} R}(x_{\alpha})} u_{\alpha}^{2^*} \, \ud  v_{g} \leq 1.
$$
Sending $\alpha$ to $\infty,$ we have
$$
\int_{B_{R}(0)} U^{2^*}\, \ud x \leq 1,
$$
which violates \eqref{eq:T1}. We have thus established Lemma \ref{lem:lemT=0}.
\end{proof}

It follows from Lemma \ref{lem:lemT=0} that
$$
\lim _{\alpha \rightarrow \infty} \operatorname{dist}(x_{\alpha}, \partial M) / \mu_{\alpha}=0,\quad v_{\alpha} \rightarrow U \text { in } C_{l o c}^{1}(\overline{\mathbb{R}_{+}^{n}}),
$$
where
$$
U(x)=\Big(\frac{1}{1+c(n)|x|^{2}}\Big)^{\frac{n-2}{2}},
$$
and $c(n)=1 /(2^{2 / n}(n-2) n S)$. A direct calculation, or similar argument as above, yields
$$
\int_{\mathbb{R}_{+}^{n}} U^{2^*}\, \ud x=1.
$$
Denote
$$
\begin{aligned}
U_{x_{0}, \lambda}(x) :&=\lambda^{-\frac{n-2}{2}} U((x-x_{0}) / \lambda) \\
&=\Big(\frac{\lambda}{\lambda^2+c(n)|x-x_{0}|^{2}}\Big)^{\frac{n-2}{2}},
\end{aligned}
$$
where $x_{0} \in \mathbb{R}_+^{n}$ and $\lambda>0$. For brevity, we write $U_{0, \lambda}$ as $U_{\lambda}$. Hence, $U_{1}=U$.

\begin{proposition}\label{pro:energy convergence}
For every $\delta>0$,
$$
\lim _{\alpha \rightarrow\infty} \int_{B_{\delta}^+}\{|\nabla_{g}(u_{\alpha}-U_{\mu_{\alpha}})|^{2}+|u_{\alpha}-U_{\mu_{\alpha}}|^{2^*}\} \, \ud  v_{g}=0.
$$
\end{proposition}

\begin{proof}
We only prove that
$$
\lim _{\alpha \rightarrow \infty} \int_{B_{\delta}^+}|\nabla_{g}(u_{\alpha}-U_{\mu_{\alpha}})|^{2} \, \ud  v_{g}=0,
$$
and the other can be proved similarly.

For every given $\varepsilon>0,$ one can find $\alpha_{0}>0$ such that for all $\alpha \geq \alpha_{0}$,
$$
\int_{M}|\nabla_{g} u_{\alpha}|^{2} \, \ud  v_{g} \leq \frac{1}{2^{2 / n} S}+\varepsilon
$$
because $\lim \limits_{\alpha \rightarrow \infty} \int_{M}|\nabla_{g} u_{\alpha}|^{2} \, \ud  v_{g}=\frac{1}{2^{2 / n} S}.$ Since $\int_{\mathbb{R}_{+}^{n}}|\nabla U_1|^{2}\, \ud x=\frac{1}{2^{2 / n} S},$ we can choose $R>0$ such that
$$
\int_{\mathbb{R}_{+}^{n} \backslash B_{R}^{+}}|\nabla U_1|^{2}\, \ud x \leq \varepsilon.
$$
Note that
$$
\int_{B_{\delta / \mu_{\alpha}}^{+}}|\nabla_{g_{\alpha}} v_{\alpha}|^{2} \, \ud  v_{g_{\alpha}}=\int_{B_{\delta}^{+}}|\nabla_{g} u_{\alpha}|^{2} \, \ud  v_{g} \leq \frac{1}{2^{2 / n} S}+\varepsilon.
$$
Hence
$$
\begin{aligned}
&\int_{B_{\delta}^{+}}|\nabla_{g}(u_{\alpha}-U_{\mu_{\alpha}})|^{2} \, \ud  v_{g} \\
=&\int_{B_{\delta / \mu_{\alpha}}^{+}}|\nabla_{g_{\alpha}}(v_{\alpha}-U_1)|^{2} \, \ud  v_{g_{\alpha}} \\
=&\int_{B_{R}^{+}}|\nabla_{g_{\alpha}}(v_{\alpha}-U_1)|^{2} \, \ud  v_{g_{\alpha}}+\int_{B_{\delta / \mu_{\alpha}}^{+} \backslash \overline{B}_{R}^{+}}|\nabla_{g_{\alpha}}(v_{\alpha}-U_1)|^{2} \, \ud  v_{g_{\alpha}} \\
\leq& 2 \varepsilon+2 \int_{B_{\delta / \mu_{\alpha}}^{+} \backslash \overline{B}_{R}^{+}}|\nabla_{g_{\alpha}} v_{\alpha}|^{2} \, \ud  v_{g_{\alpha}}+2 \int_{B_{\delta / \mu_{\alpha}}^{+} \backslash \overline{B}_{R}^{+}}|\nabla_{g_{\alpha}} U_1|^{2} \, \ud  v_{g_{\alpha}} \\
\leq& 10 \varepsilon,
\end{aligned}
$$
where we used $\|v_{\alpha}-U_1\|_{C^{1}(\overline{B}_{R}^{+})} \leq \varepsilon$ for large $\alpha$ and
$$
\begin{aligned}
&\int_{B_{\delta / \mu_{\alpha}}^{+} \backslash \overline{B}_{R}^{+}}|\nabla_{g_{\alpha}} v_{\alpha}|^{2} \, \ud  v_{g_{\alpha}}+\int_{B_{\delta / \mu_{\alpha}}^{+}\backslash \overline{B}_{R}^{+}}|\nabla_{g_{\alpha}} U_1|^{2} \, \ud  v_{g_{\alpha}} \\
\leq& \frac{1}{2^{2 / n} S}+\varepsilon-\int_{B_{R}^{+}}|\nabla_{g_{\alpha}} v_{\alpha}|^{2} \, \ud  v_{g_{\alpha}}+\frac{1}{2^{2 / n} S}+\varepsilon-\int_{B_{R}^{+}}|\nabla_{g_{\alpha}} U_{1}|^{2} \, \ud  v_{g_{\alpha}} \\
\leq& 4 \varepsilon.
\end{aligned}
$$
\end{proof}

Let $\hat{g}=\varphi^{4 /(n-2)} g$ for some positive function $\varphi$ being chosen later, then by the conformal invariance,
\be\label{eq:conformal invariance}
\left\{\begin{array}{ll}
\Delta_{\hat{g}} (\psi/\varphi)-\frac{n-2}{4(n-1)} R_{\hat{g}} (\psi/\varphi)=\varphi^{-\frac{n+2}{n-2}} (\Delta_{g} \psi-\frac{n-2}{4(n-1)} R_{g} \psi) & \text { in } M, \\
\frac{\partial_{\hat{g}}(\psi/\varphi)}{\partial \nu}+\frac{n-2}{2} h_{\hat{g}} (\psi/\varphi)=\varphi^{-\frac{n}{n-2}}(\frac{\partial_{g} \psi}{\partial \nu}+\frac{n-2}{2} h_{g} \psi) & \text { on } \partial M.
\end{array}\right.
\ee
Let $\psi=u_{\alpha}$ in \eqref{eq:conformal invariance} and write $w_{\alpha}=u_{\alpha} / \varphi$. We have
\be\label{eq:conformal invariance1}
\left\{\begin{array}{ll}
\Delta_{g} u_{\alpha}-\frac{n-2}{4(n-1)} R_{g} u_{\alpha}=\varphi^{\frac{n+2}{n-2}}(\Delta_{\hat{g}} w_{\alpha}-\frac{n-2}{4(n-1)} R_{\hat{g}} w_{\alpha}) & \text { in } M, \\
\frac{\partial_{g} u_{\alpha}}{\partial \nu}+\frac{n-2}{2} h_{g} u_{\alpha}=\varphi^{\frac{n}{n-2}}(\frac{\partial_{\hat{g}} w_{\alpha}}{\partial \nu}+\frac{n-2}{2} h_{\hat{g}} w_{\alpha}) & \text { on } \partial M.
\end{array}\right.
\ee
Let $\psi=\varphi$ in \eqref{eq:conformal invariance}. We get
\be\label{eq:conformal invariance2}
\left\{\begin{array}{ll}
-\frac{n-2}{4(n-1)} R_{\hat{g}} \varphi^{\frac{n+2}{n-2}}=\Delta_{g} \varphi-\frac{n-2}{4(n-1)} R_{g} \varphi & \text { in } M, \\
\frac{n-2}{2} h_{\hat{g}} \varphi^{\frac{n}{n-2}}=\frac{\partial_{g} \varphi}{\partial \nu}+\frac{n-2}{2} h_{g} \varphi & \text { on } \partial M.
\end{array}\right.
\ee
Combining \eqref{eq:E-L}, \eqref{eq:conformal invariance1} and \eqref{eq:conformal invariance2}, we reach the following equation
\be\label{eq:walpha}
\left\{\begin{array}{ll}
-\Delta_{\hat{g}} w_{\alpha}=\ell_{\alpha} w_{\alpha}^{2^*-1}+\varphi^{-\frac{n+2}{n-2}}w_{\alpha} \Delta_{g} \varphi   & \text { in } M, \\
\frac{\partial_{\hat{g}} w_{\alpha}}{\partial \nu}=-\varphi^{-\frac{n}{n-2}}w_{\alpha}(\frac{\partial_{g} \varphi}{\partial \nu}+\frac{n-2}{2} h_{g} \varphi)-\alpha \|u_{\alpha}\|_{L^r(\partial M)}^{2-r}\varphi^{-\frac{n}{n-2}}u_{\alpha}^{r-1} & \text { on } \partial M.
\end{array}\right.
\ee

We will choose appropriate $\varphi=\varphi_{\alpha}$ to simplify \eqref{eq:walpha} and then apply the Moser iteration technique to show that $w_{\alpha}$ is bounded above by some constant independent of $\alpha$.

Let
$$
b_{\alpha}:=\left\{\begin{array}{ll}
\min \{\frac{n-2}{2} h_{g}+\alpha\Big(\frac{\|u_{\alpha}\|_{L^{r}(\partial M)}}{u_{\alpha}}\Big)^{2-r}, 1\} & \text { if } u_{\alpha} \neq 0, \\
1 & \text { if } u_{\alpha}=0.
\end{array}\right.
$$
Note that $b_{\alpha}$ is Lipschitz on $\partial M$ (with Lipschitz constant depending on $\alpha$) and it is uniformly bounded:
$$
-\frac{n-2}{2}\|h_{g}\|_{L^\infty(\partial M)} \leq b_{\alpha} \leq 1.
$$

\begin{lemma}
If $x_{\alpha}\in \partial M$, then there exists a unique weak solution $G_{\alpha}$ of
\be\label{eq:GreenFunction1}
\left\{\begin{array}{ll}
-\Delta_{g} G_{\alpha}=0 & \text { in } M, \\
\frac{\partial_{g} G_{\alpha}}{\partial \nu}+b_{\alpha} G_{\alpha}=\delta_{x_{\alpha}} & \text { on } \partial M,
\end{array}\right.
\ee
satisfying $G_{\alpha} \in W^{1, p}(M)$ for $1 \leq p<\frac{n}{n-1},$ where $\delta_{x_{\alpha}}$ is the delta measure centered at $x_{\alpha}$. Moreover, $G_{\alpha} \in C_{{loc}}^{1}(\overline{M} \backslash\{x_{\alpha}\})$, and
\be\label{eq:GreenFunction2}
C^{-1} \operatorname{dist}_{g}(x, x_{\alpha})^{2-n} \leq G_{\alpha}(x) \leq C\operatorname{dist}_{g}(x, x_{\alpha})^{2-n} \quad \text { for all } x \in \overline{M},
\ee
where $C>0$ depends only on $M, g$.
\end{lemma}

\begin{proof}
We claim that
\be\label{eq:GreenFunction3}
\lim _{\alpha \rightarrow \infty} \text {vol}_{g}\{b_{\alpha}<\frac{1}{2}\}=0.
\ee
Indeed, for every measurable set $E \subset \subset \partial M \cap\{u_{\alpha}>0\},$ we have
$$
\operatorname{vol}_{g}(E)=\int_{E} \, \ud  s_{g}=\int_{E} u_{\alpha}^{r / 2} u_{\alpha}^{-r / 2} \, \ud  s_{g} \leq\|u_{\alpha}\|_{L^{r}(E)}^{r / 2}\|u_{\alpha}^{-1}\|_{L^{r}(E)}^{r / 2}.
$$
It follows that
$$
\begin{aligned}
\|(\|u_{\alpha}\|_{L^{r}(\partial M)} u_{\alpha}^{-1})^{2-r}\|_{L^{r /(2-r)}(E)} &=\|u_{\alpha}\|_{L^{r}(\partial M)}^{2-r}\|u_{\alpha}^{-1}\|_{L^{r}(E)}^{2-r} \\
& \geq(\operatorname{vol}_{g}(E))^{2(2-r) / r}.
\end{aligned}
$$
Note that
$$
\alpha(\|u_{\alpha}\|_{L^{r}(\partial M)} u_{\alpha}^{-1})^{2-r}<\frac{1}{2}(1+(n-2)|h_{g}|)
$$
if $b_{\alpha}<1 / 2$. Since $\{b_{\alpha}<{1}/{2}\} \subset \subset \partial M \cap\{u_{\alpha}>0\}$,
$$
(\operatorname{vol}_{g}\{b_{\alpha}<\frac{1}{2}\})^{2(2-r) / r} \leq \frac{C}{\alpha}.
$$
The claim is verified.

Notice that $b_{\alpha}$ is uniformly bounded and Lipschitz. Using H\"older's inequality, \eqref{eq:lemthm0.1} and \eqref{eq:GreenFunction3}, we have
$$
\begin{aligned}
&\int_{M}|\nabla_{g} u|^{2}\, \ud  v_{g}+\int_{\partial M} b_{\alpha} u^{2} \, \ud  s_{g} \\
\geq& \int_{M}|\nabla_{g} u|^{2} \, \ud  v_{g}+\frac{1}{2} \int_{\partial M} u^{2} \, \ud  s_{g}-\int_{\partial M}(b_{\alpha}-\frac{1}{2})^{-} u^{2} \, \ud  s_{g} \\
\geq& \int_{M}|\nabla_{g} u|^{2} \, \ud  v_{g}+\frac{1}{2} \int_{\partial M} u^{2} \, \ud  s_{g}-\|(b_{\alpha}-\frac{1}{2})^{-}\|_{L^{n-1}(\partial M)}\|u\|_{L^{2(n-1)/(n-2)}(\partial M)}^{2} \\
\geq& \int_{M}|\nabla_{g} u|^{2}\, \ud  v_{g}+\frac{1}{2} \int_{\partial M} u^{2} \, \ud  s_{g}-C\|(b_{\alpha}-\frac{1}{2})^{-}\|_{L^{n-1}(\partial M)}\Big(\int_{M}|\nabla_{g} u|^{2} \, \ud  v_{g}+\int_{\partial M} u^{2}\, \ud  s_{g}\Big) \\
\geq& \frac{1}{4} \int_{M}|\nabla_{g} u|^{2}\, \ud  v_{g}+\frac{1}{4} \int_{\partial M} u^{2} \, \ud  s_{g}.
\end{aligned}
$$
Thus, it is coercive. Then it follows from standard elliptic theory (see, e.g., \cite{BrezisStraussSemi1973}\cite{KenigPipherThe1993}) that there exists a unique solution of \eqref{eq:GreenFunction1} in $W^{1, p}(M)$ for $1 \leq p<\frac{n}{n-1}$ satisfying \eqref{eq:GreenFunction2}.
\end{proof}

Similarly, we have the following lemma:
\begin{lemma}
If $x_{\alpha}\in M$, then there exists a unique weak solution $G_{\alpha}$ of
\be\label{eq:GreenFunction1'}
\left\{\begin{array}{ll}
-\Delta_{g} G_{\alpha}=\delta_{x_{\alpha}} & \text { in } M, \\
\frac{\partial_{g} G_{\alpha}}{\partial \nu}+b_{\alpha} G_{\alpha}=0 & \text { on } \partial M,
\end{array}\right.
\ee
satisfying $G_{\alpha} \in W^{1, p}(M)$ for $1 \leq p<\frac{n}{n-1},$ where $\delta_{x_{\alpha}}$ is the delta measure centered at $x_{\alpha}$. Moreover, $G_{\alpha} \in C_{{loc}}^{1}(\overline{M} \backslash\{x_{\alpha}\})$ and
\be\label{eq:GreenFunction2'}
C^{-1} \operatorname{dist}_{g}(x, x_{\alpha})^{2-n} \leq G_{\alpha}(x) \leq C\operatorname{dist}_{g}(x, x_{\alpha})^{2-n} \quad \text { for all } x \in \overline{M},
\ee
where $C>0$ depends only on $M, g$.
\end{lemma}

Set
\be\label{eq:varphialpha}
\varphi_{\alpha}=\mu_{\alpha}^{(n-2) / 2} G_{\alpha}.
\ee
Now we have the following crucial estimate.

\begin{proposition}\label{pro:proMoser}
There exists some constant $C$ independent of $\alpha$ such that for all $\alpha \geq 1$,
\be\label{eq:Moser}
u_{\alpha} / \varphi_{\alpha} \leq C\quad \text { for }\quad x \in \overline{M}.
\ee
\end{proposition}

\begin{proof}
Set $w_{\alpha}=u_{\alpha} / \varphi_{\alpha}$ and $\hat{g}=\varphi_{\alpha}^{4/(n-2)} g .$ By our choice of $G_{\alpha}$ and \eqref{eq:walpha} we have, for large $\alpha$,
\be\label{eq:walpha<}
\left\{\begin{array}{ll}
-\Delta_{\hat{g}} w_{\alpha}=\ell_{\alpha} w_{\alpha}^{2^*-1} & \text { in } M, \\
\frac{\partial_{\hat{g}} w_{\alpha}}{\partial \nu} \leq 0 & \text { on } \partial M \backslash\{x_{\alpha}\},
\end{array}\right.
\ee
or
\be\label{eq:walpha<'}
\left\{\begin{array}{ll}
-\Delta_{\hat{g}} w_{\alpha}=\ell_{\alpha} w_{\alpha}^{2^*-1} & \text { in } M\backslash\{x_{\alpha}\}, \\
\frac{\partial_{\hat{g}} w_{\alpha}}{\partial \nu} \leq 0 & \text { on } \partial M .
\end{array}\right.
\ee
Then the Moser iterations procedure on \cite{LiZhuSharp1997} implies that
$$
\|w_{\alpha}\|_{L^{\infty}(M \backslash B_{\mu_{\alpha}}(x_{\alpha}))} \leq C.
$$
Recall that $v_{\alpha} \rightarrow U_1$ in $C_{l o c}^{2},$ from which we also have $\|w_{\alpha}\|_{L^{\infty}(B_{\mu_{\alpha}}(x_{\alpha}))} \leq C.$ This finishes the proof.
\end{proof}

\begin{corollary}\label{cor:Pointwise Estimation}
For all large $\alpha$,
$$
u_{\alpha}(x) \leq C \mu_{\alpha}^{(n-2)/2} \operatorname{dist}_{g}(x, x_{\alpha})^{2-n} \quad \text { for all } x \in \overline{M},
$$
where $C>0$ depends only on $M, g$.
\end{corollary}

\begin{proof}
It follows immediately from Proposition \ref{pro:proMoser} and \eqref{eq:GreenFunction2}.
\end{proof}

\begin{corollary}\label{cor:cor2.9}
For any small $\delta>0,$ there exists a constant $C>0$ depending only on $M, g, \delta$ such that
$$
\int_{M \backslash B_{\delta / 2}(x_{\alpha})}|\nabla_{g} u_{\alpha}|^{2}\, \ud v_g \leq C \mu_{\alpha}^{n-2}.
$$
Consequently, we can select $\delta_{\alpha} \in[\delta / 2, \delta]$ such that
$$
\int_{\partial B_{\delta_{\alpha}}(x_{\alpha})}|\nabla_{g} u_{\alpha}|^{2}\, \ud v_g \leq C \mu_{\alpha}^{n-2}.
$$
\end{corollary}

\begin{proof}
Let $\eta$ be a cut-off function satisfying $\operatorname{supp}(\eta) \subset M \backslash B_{\delta / 2}(x_{\alpha})$ and $\eta \equiv 1$ in $M \backslash B_{\delta}(x_{\alpha}) .$ Multiplying the Euler-Lagrange equation \eqref{eq:E-L} by $\eta^{2} u_{\alpha}$ and integrating by parts, we have
$$
\int_{M} \nabla_{g} u_{\alpha} \nabla_{g}(\eta^{2} u_{\alpha}) \, \ud  v_{g} \leq \ell_{\alpha} \int_{M} \eta^{2} u_{\alpha}^{2^*} \, \ud  v_{g}-\frac{n-2}{2} \int_{\partial M} h_g \eta^{2} u_{\alpha}^{2} \, \ud  s_{g}.
$$
It follows that
$$
\int_{M} \eta^{2}|\nabla_{g} u_{\alpha}|^{2} \, \ud  v_{g} \leq C \int_{\partial M}\eta^{2} u_{\alpha}^{2} \, \ud  s_{g}+C\int_{M} \eta^{2} u_{\alpha}^{2^*} \, \ud  v_{g}+C \int_{M}|\nabla_{g} \eta|^{2} u_{\alpha}^{2} \, \ud  v_{g}.
$$
Therefore, this corollary follows immediately from Corollary \ref{cor:Pointwise Estimation}.
\end{proof}

\section{Energy estimates}

Let $Q_{\alpha} \in \partial M$ be the closest point on $\partial M$ to $x_{\alpha}$. For some small $\delta_{0}$ to be determined in Lemma \ref{lem:lemcoercive}, let $\psi_{\alpha} \in C^{\infty}(\overline{M})$ satisfy $\psi_{\alpha}(Q_{\alpha})=1$, $1/2 \leq \psi_{\alpha} \leq 2$, $\|\psi_{\alpha}\|_{C^{2}(\overline{M})} \leq C,$ and
$$
\left\{\begin{array}{ll}
\Delta_{g} \psi_{\alpha}=0 & \text { in } B_{2 \delta_{0}}^{+}, \\
\frac{\partial_{g} \psi_{\alpha}}{\partial \nu}+\frac{n-2}{2} h_{g} \psi_{\alpha}=0 & \text { on } \partial^{\prime} B_{2 \delta_{0}}^{+}.
\end{array}\right.
$$
Here we used the Fermi coordinate with respect to metric $g$ centered at $Q_{\alpha}$. In this coordinate system,
$$
\sum_{1 \leq i, j \leq n} g_{i j}(x) \, \ud x_{i} \, \ud x_{j}= \ud x_{n}^{2}+\sum_{1 \leq i, j \leq n-1} g_{i j}(x)\, \ud x_{i} \, \ud x_{j}.
$$
Moreover, $g^{i j}$ has the following Taylor expansion near $\partial M$:

\begin{lemma}\label{lem:Fermi}
For $\{x_{k}\}_{k=1, \cdots, n}$ small,
$$
\begin{aligned}
&g^{i j}(x)=\delta^{i j}+2 h^{i j}(x^{\prime}, 0) x_{n}+O(|x|^{2}), \\
&g^{i j} \Gamma_{i j}^{k}=O(|x|),
\end{aligned}
$$
where $i, j=1, \cdots, n-1$, $h_{i j}$ is the second fundamental form of $\partial M$.
\end{lemma}

\begin{proof}
The proof can be found in \cite[Lemma 3.2]{EscobarConformal1992}.
\end{proof}

Set $\hat{g}=\psi_{\alpha}^{4 /(n-2)} g .$ It is easy to see that $h_{\hat{g}}=0$ on $\partial^{\prime} B_{2\delta_{0}}^{+} .$ Hence, by \eqref{eq:walpha}, $u_{\alpha} / \psi_{\alpha}$ satisfy
\be\label{eq:u/psi}
\left\{\begin{array}{ll}
-\Delta_{\hat{g}} \frac{u_{\alpha}}{\psi_{\alpha}}=\ell_{\alpha} (\frac{u_{\alpha}}{\psi_{\alpha}})^{2^*-1} & \text { in } B_{2 \delta_{0}}^{+}, \\
\frac{\partial_{\hat{g}}}{\partial \nu} \frac{u_{\alpha}}{\psi_{\alpha}}+\alpha\|u_{\alpha}\|_{L^{r}(\partial M)}^{2-r} \psi_{\alpha}^{-\frac{n}{n-2}} u_{\alpha}^{r-1}=0 & \text { on } \partial^{\prime} B_{2 \delta_{0}}^{+}.
\end{array}\right.
\ee
From now on, we use the Fermi coordinate with respect to the metric $\hat{g}$ centered at $Q_{\alpha}$. It follows from Proposition \ref{pro:energy convergence} and Taylor expansion that
\be\label{eq:3.2}
\lim _{\alpha \rightarrow\infty} \int_{B_{\delta}^+}\Big\{|\nabla_{\hat{g}}(\frac{u_{\alpha}}{\psi_{\alpha}}-U_{\mu_{\alpha}})|^{2}+|\frac{u_{\alpha}}{\psi_{\alpha}}-
U_{\mu_{\alpha}}|^{2^*}\Big\} \, \ud  v_{\hat{g}}=0.
\ee
As in Corollary \ref{cor:cor2.9}, we can select $\delta_{\alpha} \in[\delta_{0} / 2, \delta_{0}]$ such that
\be\label{eq:3.3}
\int_{\partial^{\prime \prime} B^+_{\delta_{\alpha}}}|\nabla_{\hat{g}}(\frac{u_{\alpha}}{\psi_{\alpha}})|^{2} \, \ud  v_{\hat{g}} \leq C \mu_{\alpha}^{n-2},
\ee
where $C>0$ is independent of $\alpha .$

Let $h_{\xi, \lambda}(z)$, $z \in \overline{{B}_{\delta_{\alpha}}^{+}},$ be the solution of
\be\label{eq:equ1}
\left\{\begin{array}{ll}
\Delta_{\hat{g}} h_{\xi, \lambda}=0 & \text { in } B_{\delta_{\alpha}}^{+}, \\
h_{\xi, \lambda}=U_{\xi, \lambda} & \text { on } \partial^{\prime \prime} B_{\delta_{\alpha}}^{+}, \\
\frac{\partial_{\hat{g}} h_{\xi, \lambda}}{\partial \nu}=0 & \text { on } \partial^{\prime} B_{\delta_{\alpha}}^{+},
\end{array}\right.
\ee
with parameters $\xi \in \partial^{\prime} B_{\mu_{\alpha}\delta_{\alpha}  / 2}^{+}$ and $\lambda>0,$ and let $\chi_{\alpha}(z)$ be the solution of
\be\label{eq:equ2}
\left\{\begin{array}{ll}
\Delta_{\hat{g}} \chi_{\alpha}=0 & \text { in } B_{\delta_{\alpha}}^{+}, \\
\chi_{\alpha}=\frac{u_{\alpha}}{\psi_{\alpha}} & \text { on } \partial^{\prime \prime} B_{\delta_{\alpha}}^+, \\
\frac{\partial_{\hat{g}} \chi_{\alpha}}{\partial \nu}=0 & \text { on } \partial^{\prime} B_{\delta_{\alpha}}^{+}.
\end{array}\right.
\ee
Then $U_{\xi, \lambda}-h_{\xi, \lambda} \in H_{0, L}(B_{\delta_{\alpha}}^{+})$ and $u_{\alpha} / \psi_{\alpha}-\chi_{\alpha} \in H_{0, L}(B_{\delta_{\alpha}}^{+})$ are the projections of $U_{\xi, \lambda}$ and $u_{\alpha} / \psi_{\alpha}$ on $H_{0, L}(B_{\delta_{\alpha}}^+),$ respectively. Here
$$
H_{0, L}(B_{\delta_{\alpha}}^{+}):=\{u \in H^{1}(B_{\delta_{\alpha}}^{+}): u=0 \text { on } \partial^{\prime \prime} B_{\delta_{\alpha}}^{+} \text { in trace sense}\}
$$
is a Hilbert space with the inner product $\langle u, v\rangle_{\hat{g}}:=\int_{B_{\delta_{\alpha}}^{+}} \nabla_{\hat{g}} u \nabla_{\hat{g}} v \, \ud  v_{\hat{g}} .$ Denote $\|u\|=\sqrt{\langle u, u\rangle_{\hat{g}}},$ which is a norm for $u \in H_{0, L}(B_{\delta_{\alpha}}^{+})$.

Set
$$
\sigma_{\xi, \lambda}=U_{\xi, \lambda}-h_{\xi, \lambda},
$$
which satisfies $\sigma_{\xi, \lambda} \leq U_{\xi, \lambda}$ and
\be\label{eq:equ3}
\left\{\begin{array}{ll}
\Delta_{\hat{g}} \sigma_{\xi, \lambda}=\Delta_{\hat{g}} U_{\xi, \lambda} & \text { in } B_{\delta_{\alpha}}^{+}, \\
\sigma_{\xi, \lambda}=0 & \text { on } \partial^{\prime \prime} B_{\delta_{\alpha}}^{+}, \\
\frac{\partial_{\hat{g}} \sigma_{\xi, \lambda}}{\partial \nu}=\frac{\partial_{\hat{g}} U_{\xi, \lambda}}{\partial \nu} & \text { on } \partial^{\prime} B_{\delta_{\alpha}}^{+}.
\end{array}\right.
\ee
Let $(t_{\alpha}, \xi_{\alpha}, \lambda_{\alpha}) \in[\frac{1}{2}, \frac{3}{2}] \times \overline{\partial^{\prime} B_{\mu_{\alpha}\delta_{\alpha}  / 2}^{+}} \times[\frac{\mu_{\alpha}}{2}, \frac{3 \mu_{\alpha}}{2}]$ be such that
$$
\begin{aligned}
&\|\frac{u_{\alpha}}{\psi_{\alpha}}-\chi_{\alpha}-t_{\alpha} \sigma_{\xi_{\alpha}, \lambda_{\alpha}}\| \\
=&\min \Big\{\|\frac{u_{\alpha}}{\psi_{\alpha}}-\chi_{\alpha}-t \sigma_{\xi, \lambda}\|:(t, \xi, \lambda) \in[\frac{1}{2}, \frac{3}{2}] \times \overline{\partial^{\prime} B_{\mu_{\alpha}\delta_{\alpha}  / 2}^{+}} \times[\frac{\mu_{\alpha}}{2}, \frac{3 \mu_{\alpha}}{2}]\Big\}.
\end{aligned}
$$
Denote
$$
w_{\alpha}=\frac{u_{\alpha}}{\psi_{\alpha}}-\chi_{\alpha}-t_{\alpha} \sigma_{\xi_{\alpha}, \lambda_{\alpha}}.
$$
The main result in this section is the following estimate for $w_{\alpha}$:

\begin{proposition}\label{pro:proEnergyEstimation}
Assume as above that we have
$$
\|w_{\alpha}\|\leq C\Big\{\mu_{\alpha}^2\|U_{1}\|_{L^{2^{*\prime}}(B_{\mu_{\alpha}^{-1}}^{+})}+\varepsilon_{\alpha}\|U_{1}^{r-1}\|_{L^{r}(\partial^{\prime} B_{\mu_{\alpha}^{-1}}^{+})}+\mu_{\alpha}^{n-2}\|U_{1}^{2^*-2}\|_{L^{2^{*\prime}}(B_{\mu_{\alpha}^{-1}}^{+})}\Big\},
$$
where $\varepsilon_{\alpha}$ is given in \eqref{eq:varepsilonalpha}.
\end{proposition}

Define
$$
W_{\alpha}=\{w \in H_{0, L}(B_{\delta_{\alpha}}^{+}):\langle\sigma_{\xi_{\alpha}, \lambda_{\alpha}}, w\rangle_{\hat{g}}=0 \text { and }\langle u, w\rangle_{\hat{g}}=0 \text { for all } u \in E_{\alpha}\},
$$
where $E_{\alpha} \subset H_{0, L}(B_{\delta_{\alpha}}^{+})$ is the tangent space at $\sigma_{\xi_{\alpha}, \lambda_{\alpha}}$ of the $n$-dimensional surface $\{\sigma_{\xi, \lambda}: \xi \in \partial^{\prime} B^+_{\mu_{\alpha} \delta_{\alpha}}, \lambda>0\} \subset H_{0, L}(B_{\delta_{\alpha}}^{+})$. More explicitly,
$$
E_{\alpha}=\operatorname{span}\Big\{\frac{\partial \sigma_{\xi, \lambda_{\alpha}}}{\partial \xi}\Big|_{\xi=\xi_{\alpha}}, \frac{\partial \sigma_{\xi_{\alpha}, \lambda}}{\partial \lambda}\Big|_{\lambda=\lambda_{\alpha}}\Big\}.
$$

To simplify notation, henceforth we denote:
$$
h_{\alpha}=h_{\xi_{\alpha}, \lambda_{\alpha}}\quad \text{ and } \quad\sigma_{\alpha}=\sigma_{\xi_{\alpha}, \lambda_{\alpha}}.
$$

\begin{lemma}\label{lem:lem3.1}
We have:

(i) $\|\nabla_{\hat{g}} h_{\alpha}\|_{L^{2}(B^+_{\delta_{\alpha}})}+\|h_{\alpha}\|_{L^{\infty}(B^+_{\delta_{\alpha}})} \leq C \mu_{\alpha}^{(n-2) / 2}$,

(ii) $\|\nabla_{\hat{g}} \chi_{\alpha}\|_{L^{2}(B^+_{\delta_{\alpha}})}+\|\chi_{\alpha}\|_{L^{\infty}(B^+_{\delta_{\alpha}})} \leq C \mu_{\alpha}^{(n-2) / 2}$,\\
for some positive constant $C$ independent of $\alpha,$ and

(iii) $\|w_{\alpha}\| \rightarrow 0$,

(iv) $t_{\alpha} \rightarrow 1$,

(v) $\mu_{\alpha}^{-1}|\xi_{\alpha}| \rightarrow 0$,

(vi) $\mu_{\alpha}^{-1} \lambda_{\alpha} \rightarrow 1$,\\
as $\alpha \rightarrow \infty$. Furthermore, $w_{\alpha} \in W_{\alpha}$.
\end{lemma}

\begin{proof}
Let $\eta \in H^{1}(B^+_{\delta_{\alpha}})$ be an extension of $U_{\xi_{\alpha}, \lambda_{\alpha}}|_{\partial^{\prime \prime} B_{\delta_{\alpha}}^{+}}$ such that
$$
\|\eta\|_{H^{1}(B^+_{\delta_{\alpha}})}^{2} \leq C\Big(\int_{\partial^{\prime \prime} B_{\delta_{\alpha}}^{+}}|\nabla U_{\xi_{\alpha}, \lambda_{\alpha}}|^{2}+U_{\xi_{\alpha}, \lambda_{\alpha}}^{2}\Big),
$$
where $C>0$ is independent of $\alpha$. Multiplying \eqref{eq:equ1} by $h_{\alpha}-\eta \in H_{0, L}(B^+_{\delta_{\alpha}})$ and integrating by parts we have
$$
0=\int_{B^+_{\delta_{\alpha}}} \nabla_{\hat{g}} h_{\alpha} \nabla_{\hat{g}}(h_{\alpha}-\eta) \, \ud v_{\hat{g}} \geq \frac{1}{2}(\|\nabla_{\hat{g}} h_{\alpha}\|_{L^{2}(B^+_{\delta \alpha})}^{2}-\|\nabla_{\hat{g}} \eta\|_{L^{2}(B^+_{\delta \alpha})}^{2}).
$$
Thus we obtained the $L^{2}$-estimate for $\nabla_{\hat{g}} h_{\alpha}$. The $L^{\infty}$-estimate for $h_{\alpha}$ follow easily from the maximum principle. Hence, we verified (i). Similarly, we can verify (ii) by taking into account Corollary \ref{cor:Pointwise Estimation} and \eqref{eq:3.3}.

By the definition of $t_{\alpha}$ and $\sigma_{\alpha}$,
$$
\begin{aligned}
\|t_{\alpha} \sigma_{\alpha}-\sigma_{0, \mu_{\alpha}}\| & \leq\|\frac{u_{\alpha}}{\psi_{\alpha}}-\chi_{\alpha}-t_{\alpha} \sigma_{\alpha}\|+\|\frac{u_{\alpha}}{\psi_{\alpha}}-\chi_{\alpha}-\sigma_{0, \mu_{\alpha}}\| \\
& \leq 2\|\frac{u_{\alpha}}{\psi_{\alpha}}-\chi_{\alpha}-\sigma_{0, \mu_{\alpha}}\| \\
& \leq 2\|\frac{u_{\alpha}}{\psi_{\alpha}}-U_{\mu_{\alpha}}\|+C \mu_{\alpha}^{(n-2) / 2} \rightarrow 0
\end{aligned}
$$
as $\alpha \rightarrow \infty,$ where we used (i), (ii), and \eqref{eq:3.2}. It follows that $\|w_{\alpha}\| \rightarrow 0$, i.e., (iii) holds, and
$$
\|t_{\alpha} U_{\xi_{\alpha}, \lambda_{\alpha}}-U_{\mu_{\alpha}}\| \leq\|t_{\alpha} \sigma_{\alpha}-\sigma_{0, \mu_{\alpha}}\|+\|t_{\alpha} h_{\xi_{\alpha}, \lambda_{\alpha}}\|+\|h_{0, \mu_{\alpha}}\| \rightarrow 0
$$
as $\alpha \rightarrow \infty$. A simple calculation yields (iv), (v), and (vi). Once we have (iv), (v), and (vi), the minimum of the norm is attained in the interior of $[\frac{1}{2}, \frac{3}{2}] \times \overline{\partial^{\prime} B_{\mu_{\alpha}\delta_{\alpha}  / 2}^{+}} \times[\frac{\mu_{\alpha}}{2}, \frac{3 \mu_{\alpha}}{2}]$. Hence, a variational argument gives $w_{\alpha} \in W_{\alpha}$.
\end{proof}

In order to estimate $w_{\alpha},$ we begin by writing an equation for $w_{\alpha}$:

\begin{lemma}\label{lem:lemwalphaEquation}
$w_{\alpha}$ satisfies
\be\label{eq:walphaEquation}
\left\{\begin{array}{ll}
-\Delta_{\hat{g}} w_{\alpha}=k_{\alpha}|\Theta_{\alpha}|^{2^*-3} \Theta_{\alpha} w_{\alpha}+b^{\prime}|\Theta_{\alpha}|^{2^*-3} w_{\alpha}^{2}+b^{\prime \prime}|w_{\alpha}|^{2^*-1}+f_{\alpha} & \text { in } B_{\delta_{\alpha}}^{+}, \\
\frac{\partial_{\hat{g}}w_{\alpha}}{\partial \nu}=-\alpha\|u_{\alpha}\|_{L^{r}(\partial M)}^{2-r} \psi_{\alpha}^{-\frac{n}{n-2}} u_{\alpha}^{r-1} & \text { on } \partial^{\prime} B_{\delta_{\alpha}}^{+},
\end{array}\right.
\ee
where
$$
\begin{aligned}
k_{\alpha}&=(2^*-1)\ell_{\alpha},\\
\Theta_{\alpha}&=t_{\alpha} \sigma_{\alpha}+\chi_{\alpha}, \\
f_{\alpha}&=\ell_{\alpha} (t_{\alpha}U_{\xi_{\alpha}, \lambda_{\alpha}})^{2^*-1}+t_{\alpha} \Delta_{\hat{g}} U_{\xi_{\alpha}, \lambda_{\alpha}}+O(\mu_{\alpha}^{(n-2) / 2}) U_{\xi_{\alpha}, \lambda_{\alpha}}^{2^*-2},
\end{aligned}
$$
and $b^{\prime}$, $b^{\prime \prime}$ are bounded functions with $b^{\prime} \equiv 0$ if $n \geq 6$.
\end{lemma}

\begin{proof}
First of all, by the definition of $w_{\alpha}$, \eqref{eq:u/psi}, \eqref{eq:equ2} and \eqref{eq:equ3},  we have
\be\label{eq:Deltawalpha}
\begin{aligned}
-\Delta_{\hat{g}} w_{\alpha}&=-\Delta_{\hat{g}}\Big(\frac{u_{\alpha}}{\psi_{\alpha}}-\chi_{\alpha}-t_{\alpha} \sigma_{\alpha}\Big)\\
&=\ell_{\alpha}(\Theta_{\alpha}+w_{\alpha})^{2^*-1}+t_{\alpha} \Delta_{\hat{g}} \sigma_{\alpha}\\
&=\ell_{\alpha}(\Theta_{\alpha}+w_{\alpha})^{2^*-1}+t_{\alpha} \Delta_{\hat{g}} U_{\xi_{\alpha}, \lambda_{\alpha}}
\end{aligned}\quad \text { in } B_{\delta_{\alpha}}^{+},
\ee
and
$$
\begin{aligned}
\frac{\partial_{\hat{g}} w_{\alpha}}{\partial \nu} &=-\alpha\|u_{\alpha}\|_{L^{r}(\partial M)}^{2-r} \psi_{\alpha}^{-\frac{n}{n-2}} u_{\alpha}^{r-1}-t_{\alpha} \frac{\partial_{\hat{g}} \sigma_{\alpha}}{\partial \nu} \\
&=-\alpha\|u_{\alpha}\|_{L^{r}(\partial M)}^{2-r} \psi_{\alpha}^{-\frac{n}{n-2}} u_{\alpha}^{r-1}-t_{\alpha} \frac{\partial_{\hat{g}} U_{\xi_{\alpha}, \lambda_{\alpha}}}{\partial \nu} \\
&=-\alpha\|u_{\alpha}\|_{L^{r}(\partial M)}^{2-r} \psi_{\alpha}^{-\frac{n}{n-2}} u_{\alpha}^{r-1}
\end{aligned}\quad \text { on } \partial^{\prime} B_{\delta_{\alpha}}^{+},
$$
where we used that $\frac{\partial_{\hat{g}} U_{\xi_{\alpha}, \lambda_{\alpha}}}{\partial \nu}=\frac{\partial U_{\xi_{\alpha}, \lambda_{\alpha}}}{\partial x_{n}}=0$ in Fermi coordinate systems. In order to simplify the right hand side in \eqref{eq:Deltawalpha}, we use the elementary expansion
$$
(x+y)^{2^*-1}=|x|^{2^*-2} x+(2^*-1)|x|^{2^*-3} x y+b^{\prime}(x, y)|x|^{2^*-3} y^{2}+b^{\prime \prime}(x, y)|y|^{2^*-1}
$$
for all $x, y \in \mathbb{R}$ such that $x+y \geq 0,$ where $b^{\prime}, b^{\prime \prime}$ are bounded functions and $b^{\prime} \equiv 0$ if $n \geq 6$. For $x=\Theta_{\alpha}$, $y=w_{\alpha},$ we obtain
$$
(\Theta_{\alpha}+w_{\alpha})^{2^*-1}=|\Theta_{\alpha}|^{2^*-2} \Theta_{\alpha}+(2^*-1)|\Theta_{\alpha}|^{2^*-3} \Theta_{\alpha} w_{\alpha} +b^{\prime}|\Theta_{\alpha}|^{2^*-3} w_{\alpha}^{2}+b^{\prime \prime}|w_{\alpha}|^{2^*-1}.
$$
Note that $\Theta_{\alpha}=t_{\alpha} U_{\xi_{\alpha}, \lambda_{\alpha}}-t_{\alpha} h_{\alpha}+\chi_{\alpha} .$ By Lemma \ref{lem:lem3.1} and properties of $U_{\xi_{\alpha}, \lambda_{\alpha}},$ we have $|\chi_{\alpha}-t_{\alpha} h_{\alpha}| \leq C \mu_{\alpha}^{(n-2) / 2} \leq C t_{\alpha} U_{\xi_{\alpha}, \lambda_{\alpha}},$ and thus by simple calculus:
$$
\begin{aligned}
|\Theta_{\alpha}|^{2^*-2} \Theta_{\alpha}=&|t_{\alpha} U_{\xi_{\alpha}, \lambda_{\alpha}}-t_{\alpha} h_{\alpha}+\chi_{\alpha}|^{2^*-2}(t_{\alpha} U_{\xi_{\alpha}, \lambda_{\alpha}}-t_{\alpha} h_{\alpha}+\chi_{\alpha}) \\
=&(t_{\alpha} U_{\xi_{\alpha}, \lambda_{\alpha}})^{2^*-1}+(t_{\alpha} U_{\xi_{\alpha}, \lambda_{\alpha}})^{2^*-2}(\chi_{\alpha}-t_{\alpha} h_{\alpha})\\
&+(2^*-2)|t_{\alpha} U_{\xi_{\alpha}, \lambda_{\alpha}}+\theta(\chi_{\alpha}-t_{\alpha} h_{\alpha})|^{2^*-3}(\chi_{\alpha}-t_{\alpha} h_{\alpha})(t_{\alpha} U_{\xi_{\alpha}, \lambda_{\alpha}}-t_{\alpha} h_{\alpha}+\chi_{\alpha})\\
=&(t_{\alpha} U_{\xi_{\alpha}, \lambda_{\alpha}})^{2^*-1}+O(\mu_{\alpha}^{(n-2) / 2} U_{\xi_{\alpha}, \lambda_{\alpha}}^{2^*-2}),
\end{aligned}
$$
where $\theta \in(0,1)$. Inserting the above expansions into \eqref{eq:Deltawalpha}, we obtain \eqref{eq:walphaEquation}.
\end{proof}

The proof of Proposition \ref{pro:proEnergyEstimation} relies on the coercivity property as in Lemma \ref{lem:lemcoercive} below. Define
$$
Q_{\alpha}(\varphi, \psi):=\int_{B_{\delta_{\alpha}}^{+}}\{ \nabla_{\hat{g}} \varphi \nabla_{\hat{g}} \psi-k_{\alpha}|\Theta_{\alpha}|^{2^*-3} \Theta_{\alpha} \varphi \psi \}\, \ud  v_{\hat{g}}
$$
for all $\varphi, \psi \in H_{0, L}(B_{\delta_{\alpha}}^{+})$, where $k_{\alpha}$ and $\Theta_{\alpha}$ are defined in Lemma \ref{lem:lemwalphaEquation}.

\begin{lemma}\label{lem:lemcoercive}
There exist $0<\delta_{0} \ll 1$, $\alpha_{0} \geq1$, and $c_{0}>0$ independent of $\alpha$ such that
$$Q_{\alpha}(w, w) \geq c_{0} \int_{B_{\delta_{\alpha}}^{+}}|\nabla_{\hat{g}} w|^{2} \, \ud
v_{\hat{g}}, \quad \forall \, w \in W_{\alpha},\ \forall \, \alpha \geq \alpha_{0}.$$
\end{lemma}

\begin{proof}
The proof follows from Lemma \ref{lem:lem3.1} and Lemma \ref{lem:lemcoercive1}.
\end{proof}

Let ${D}^{1,2}(\mathbb{R}_{+}^{n})$ be the closure of $C_{c}^{\infty}(\mathbb{R}_{+}^{n} \cup \partial \mathbb{R}_{+}^{n})$ under the norm
$$
\|u\|_{{D}^{1,2}(\mathbb{R}_{+}^{n})}=\Big(\int_{\mathbb{R}_{+}^{n}}|\nabla u|^{2} \, \ud  y\Big)^{1 / 2}.
$$
In fact, ${D}^{1,2}(\mathbb{R}_{+}^{n})$ is a Hilbert space with the inner product
$$
\langle\varphi, \psi\rangle_{1}:=\int_{\mathbb{R}_{+}^{n}} \nabla \varphi \nabla \psi \, \ud  y
$$
for any $\varphi, \psi \in {D}^{1,2}(\mathbb{R}_{+}^{n})$. Define the functional
$$
Q_{1}(\varphi, \psi):=\int_{\mathbb{R}_{+}^{n}}\{ \nabla \varphi \nabla \psi -\frac{2^*-1}{2^{2 / n} S}U_{1}^{2^*-2} \varphi \psi \}\, \ud y
$$
for all $\varphi, \psi \in E_{1}$, where
$$
\begin{aligned}
E_{1}=\Big\{w \in {D}^{1,2}(\mathbb{R}_{+}^{n}):\, &\langle\frac{\partial U_{\xi, 1}}{\partial \xi_{i}}\Big|_{\xi=0}, w\rangle_{1}=\langle\frac{\partial U_{0, \lambda}}{\partial \lambda}\Big|_{\lambda=1}, w\rangle_{1}=\langle U_{1}, w\rangle_{1}=0, \\
&i=1,2, \cdots, n-1\Big\}.
\end{aligned}
$$
It is well-known that (see, e.g., Li \cite{LiOn1997}) there exists $c_{1}>0$ such that
\be\label{eq:coercive}
Q_{1}(w, w) \geq c_{1}\|w\|_{{D}^{1,2}(\mathbb{R}_{+}^{n})}^{2}, \quad \forall\,  w \in E_{1}.
\ee

\begin{lemma}\label{lem:lemcoercive1}
For $R>0$ and $x \in \partial \mathbb{R}^{n}_+$ with $|x| \leq R / 10$, let $h_{i j}$ be a Riemannian metric on $B_{R}^{+}(x)$, $k>0$ and $\Theta \in L^{2^*}(B_{R}^+(x)) .$ Denote $Q_{2}$ as the continuous bilinear form on $H_{0, L}(B_{R}^{+}(x)) \times H_{0, L}(B_{R}^{+}(x))$:
$$
Q_{2}(\varphi, \psi)=\int_{B_{R}^{+}(x)}\{ \nabla_{h} \varphi \nabla_{h} \psi-k|\Theta|^{2^*-3} \Theta \varphi \psi \}\, \ud v_{h}.
$$
There exists a small positive $\varepsilon_{0}$ depending only on $n$ such that if
$$
\|\Theta-U_{1}\|_{L^{2^*}(B_{R}^{+}(x))}+|k-\frac{2^*-1}{2^{2 / n} S}|+\|h_{i j}-\delta_{i j}\|_{L^{\infty}(B_{R}^{+}(x))} \leq \varepsilon_{0},
$$
then
$$
Q_{2}(w, w) \geq \frac{c_{1}}{2} \int_{B_{R}^{+}(x)}|\nabla_{h} w|^{2} \, \ud v_{h}
$$
for all $w \in E_{2}$, where
$$
\begin{aligned}
E_{2}=\Big\{w \in H_{0, L}(B_{R}^{+}(x)):\, &\langle\frac{\partial U_{\xi, 1}}{\partial \xi_{i}}\Big|_{\xi=0}, w\rangle_{1} \leq \varepsilon_{0}\|w\|_{h},\, \langle\frac{\partial U_{0, \lambda}}{\partial \lambda}\Big|_{\lambda=1}, w\rangle_{1} \leq \varepsilon_{0}\|w\|_{h}, \\
&\langle U_{1}, w\rangle_{1} \leq \varepsilon_{0}\|w\|_{h},\, i=1,2, \cdots, n-1\Big\}.
\end{aligned}
$$
\end{lemma}

\begin{proof}
From the assumptions, there exist unique $\delta_{1}, \cdots, \delta_{n+1}$ satisfying $|\delta_{j}|=O(\varepsilon_{0}\|w\|_{h})$ for every $j$ such that
$$
\tilde{w}=w-\sum_{j=1}^{n-1} \delta_{j} \frac{\partial U_{\xi, 1}}{\partial \xi_{i}}\Big|_{\xi=0}-\delta_{n} \frac{\partial U_{0, \lambda}}{\partial \lambda}\Big|_{\lambda=1}-\delta_{n+1} U_{1}
$$
belongs to $E_{1}$. It follows from \eqref{eq:coercive} that
$$
Q_{1}(\tilde{w}, \tilde{w}) \geq c_{1}\|\tilde{w}\|_{{D}^{1,2}(\mathbb{R}_{+}^{n})}.
$$
The lemma follows easily.
\end{proof}

{\bf \noindent Proof of Proposition \ref{pro:proEnergyEstimation}.}\quad Multiplying \eqref{eq:walphaEquation} by $w_{\alpha}$ and integrating over $B_{\delta_{\alpha}}^+$ we obtain
$$
Q_{\alpha}(w_{\alpha}, w_{\alpha})+o(\|w_{\alpha}\|^{2})=\int_{B_{\delta_{\alpha}}^{+}} w_{\alpha} f_{\alpha} \, \ud  v_{\hat{g}}-\int_{\partial^{\prime} B_{\delta_{\alpha}}^{+}} \alpha\|u_{\alpha}\|_{L^{r}(\partial M)}^{2-r} \psi_{\alpha}^{-\frac{n}{n-2}} u_{\alpha}^{r-1}w_{\alpha} \, \ud  s_{\hat{g}}.
$$
By H$\ddot{\text{o}}$lder's inequality, Sobolev's inequality, and Lemma \ref{lem:lemcoercive}, we have
$$
\begin{aligned}
\|w_{\alpha}\| \leq C\Big\{&\|2^{2 / n} S\Delta_{\hat{g}} U_{\xi_{\alpha}, \lambda_{\alpha}}+U_{\xi_{\alpha}, \lambda_{\alpha}}^{2^*-1}\|_{L^{2^{*\prime}}(B_{\delta_{\alpha}}^{+})}+\alpha\|u_{\alpha}\|_{L^{r}(\partial M)}^{2-r}\|u_{\alpha}^{r-1}\|_{L^{r}(\partial^{\prime} B_{\delta_{\alpha}}^{+})}\\
&+\mu_{\alpha}^{\frac{n-2}{2}}\|U_{\xi_{\alpha}, \lambda_{\alpha}}^{2^*-2}\|_{L^{2^{*\prime}}(B_{\delta_{\alpha}}^{+})}\Big\},
\end{aligned}
$$
where $2^{*\prime}=2n/(n+2)$. Note that $U_{\xi_{\alpha}, \lambda_{\alpha}}$ satisfies
\be\label{eq:eqUxilambda}
\left\{\begin{array}{ll}
-\Delta_{\hat{g}} U_{\xi_{\alpha}, \lambda_{\alpha}}=2^{-2/n}S^{-1} U_{\xi_{\alpha}, \lambda_{\alpha}}^{2^{*}-1}+O(U_{\xi_{\alpha}, \lambda_{\alpha}}) & \text { in } B_{\delta_{\alpha}}^{+}, \\
\frac{\partial_{\hat{g}} U_{\xi_{\alpha}, \lambda_{\alpha}}}{\partial \nu}=0 & \text { on } \partial^{\prime} B_{\delta_{\alpha}}^{+}.
\end{array}\right.
\ee
It follows that
$$
\|2^{2 / n} S\Delta_{\hat{g}} U_{\xi_{\alpha}, \lambda_{\alpha}}+U_{\xi_{\alpha}, \lambda_{\alpha}}^{2^*-1}\|_{L^{2^{*\prime}}(B_{\delta_{\alpha}}^{+})}\leq C\|U_{\xi_{\alpha}, \lambda_{\alpha}}\|_{L^{2^{*\prime}}(B_{\delta_{\alpha}}^{+})}\leq C \lambda_{\alpha}^2\|U_{1}\|_{L^{2^{*\prime}}(B_{\lambda_{\alpha}^{-1}}^{+})}.
$$
By Corollary \ref{cor:Pointwise Estimation}, we have $u_{\alpha} \leq C U_{\mu_{\alpha}}$. Hence
$$
\|u_{\alpha}^{r-1}\|_{L^{r}(\partial^{\prime} B_{\delta_{\alpha}}^{+})} \leq C\|U_{\mu_{\alpha}}^{r-1}\|_{L^{r}(\partial^{\prime} B_{\delta_{\alpha}}^{+})} \leq C\mu_{\alpha}^{n-1-\frac{n-2}{2} r}\|U_{1}^{r-1}\|_{L^{r}(\partial^{\prime} B_{\mu_{\alpha}^{-1}}^{+})}.
$$
Similarly, we compute:
$$
\|U_{\xi_{\alpha}, \lambda_{\alpha}}^{2^*-2}\|_{L^{2^{*\prime}}(B_{\delta_{\alpha}}^{+})} \leq C \mu_{\alpha}^{\frac{n-2}{2}}\|U_{1}^{2^*-2}\|_{L^{2^{*\prime}}(B_{\mu_{\alpha}^{-1}}^{+})}.
$$
Therefore, we obtained the estimate for $\|w_{\alpha}\|$.

\section{Proof of Theorem \ref{thm:mainthm2}}

Let
\be\label{eq:Ygu}
Y_g(u_{\alpha})=\frac{\int_{M}|\nabla_{g} u_{\alpha}|^{2} \, \ud  v_{g}+\frac{n-2}{2} \int_{\partial M} h_{g} u_{\alpha}^{2} \, \ud  s_{g}}{(\int_{M} u_{\alpha}^{2^*} \, \ud  v_{g})^{2 /2^*}}.
\ee
In this section we shall carefully exploit orthogonality in order to derive a lower bound for $Y_{g}(u_{\alpha})$. Together with the
estimates from the previous sections, it will readily imply the proof of Theorem \ref{thm:mainthm2}.

It follows from Corollary \ref{cor:Pointwise Estimation} and Corollary \ref{cor:cor2.9} that
$$
Y_g(u_{\alpha})=\frac{\int_{B_{\delta_{\alpha}}^{+}}|\nabla_{g} u_{\alpha}|^{2} \, \ud  v_{g}+\frac{n-2}{2} \int_{\partial^{\prime} B_{\delta_{\alpha}}^{+}} h_{g} u_{\alpha}^{2} \, \ud  s_{g}}{(\int_{ B_{\delta_{\alpha}}^{+}} u_{\alpha}^{2^*} \, \ud  v_{g})^{2 / 2^*}}+O(\mu_{\alpha}^{n-2}).
$$
By conformal invariance \eqref{eq:conformal invariance} and $h_{\hat{g}}=0$ on $\partial^{\prime} B_{\delta_{\alpha}}^{+},$ we have
$$
Y_g(u_{\alpha})=\frac{\int_{B_{\delta_{\alpha}}^{+}}|\nabla_{\hat{g}}(\frac{u_{\alpha}}{\psi_{\alpha}})|^{2} \, \ud  v_{\hat{g}}}{(\int_{B_{\delta_{\alpha}}^{+}}(\frac{u_{\alpha}}{\psi_{\alpha}})^{2^*} \, \ud  v_{\hat{g}})^{2 / 2^*}}+O(\mu_{\alpha}^{n-2}).
$$
By $u_{\alpha} / \psi_{\alpha}=t_{\alpha} U_{\xi_{\alpha}, \lambda_{\alpha}}-t_{\alpha} h_{\alpha}+\chi_{\alpha}+w_{\alpha}$,
$$
\int_{B_{\delta_{\alpha}}^{+}} \nabla_{\hat{g}} \chi_{\alpha} \nabla_{\hat{g}} w_{\alpha} \, \ud  v_{\hat{g}}=\int_{B_{\delta_{\alpha}}^{+}} \nabla_{\hat{g}} h_{\alpha} \nabla_{\hat{g}} w_{\alpha} \, \ud  v_{\hat{g}}=0,
$$
and the estimates in Lemma \ref{lem:lem3.1}, we have
\be\label{eq:Ygualpha}
Y_g(u_{\alpha})=F(w_{\alpha})+O(\mu_{\alpha}^{n-2}),
\ee
where
$$
F(w):=\frac{\int_{B_{\delta_{\alpha}}^{+}}|\nabla_{\hat{g}}(t_{\alpha} U_{\xi_{\alpha}, \lambda_{\alpha}}+w)|^{2} \, \ud  v_{\hat{g}}}{(\int_{B_{\delta_{\alpha}}^{+}}|t_{\alpha} U_{\xi_{\alpha}, \lambda_{\alpha}}+w|^{2^*}\, \ud  v_{\hat{g}})^{2 / 2^*}}.
$$
A Taylor expansion yields:
$$
F(w_{\alpha})=F(0)+F^{\prime}(0) w_{\alpha}+\frac{1}{2}\langle F^{\prime \prime}(0) w_{\alpha}, w_{\alpha}\rangle+o(\|w_{\alpha}\|^{2}),
$$
where $F^{\prime}$, $F^{\prime \prime}$ denote Fr\'echet derivatives. By a direct computation,
$$
\begin{aligned}
F^{\prime}(0) w_{\alpha}=& \frac{2}{(\int_{B_{\delta_{\alpha}}^{+}}|t_{\alpha} U_{\xi_{\alpha}, \lambda_{\alpha}}|^{2^*} \, \ud  v_{\hat{g}})^{2 / 2^*}}\Big\{\int_{B_{\delta_{\alpha}}^{+}} t_{\alpha} \nabla_{\hat{g}} U_{\xi_{\alpha}, \lambda_{\alpha}} \nabla_{\hat{g}} w_{\alpha} \, \ud  v_{\hat{g}}\\
&-\frac{\int_{B_{\delta_{\alpha}}^{+}}|t_{\alpha} \nabla_{\hat{g}} U_{\xi_{\alpha}, \lambda_{\alpha}}|^{2} \, \ud  v_{\hat{g}}}{\int_{B_{\delta_{\alpha}}^{+}}|t_{\alpha} U_{\xi_{\alpha}, \lambda_{\alpha}}|^{2^*} \, \ud  v_{\hat{g}}} \int_{B_{\delta_{\alpha}}^{+}}|t_{\alpha} U_{\xi_{\alpha}, \lambda_{\alpha}}|^{2^*-1} w_{\alpha} \, \ud  v_{\hat{g}}\Big\}.
\end{aligned}
$$
Recall that $\sigma_{\alpha}=U_{\xi_{\alpha}, \lambda_{\alpha}}-h_{\alpha}$ and $\int_{B_{\delta_{\alpha}}^{+}} \nabla_{\hat{g}} \sigma_{\alpha} \nabla_{\hat{g}} w_{\alpha}\, \ud  v_{\hat{g}}=\int_{B_{\delta_{\alpha}}^{+}} \nabla_{\hat{g}} h_{\alpha} \nabla_{\hat{g}} w_{\alpha}\, \ud  v_{\hat{g}}=0 .$ It
follows that $\int_{B_{\delta_{\alpha}}^{+}} \nabla_{\hat{g}} U_{\xi_{\alpha}, \lambda_{\alpha}} \nabla_{\hat{g}} w_{\alpha}\, \ud  v_{\hat{g}}=0 .$ By \eqref{eq:eqUxilambda}, we have
$$
\begin{aligned}
|F^{\prime}(0) w_{\alpha}| & \leq C|\int_{B_{\delta_{\alpha}}^{+}}|U_{\xi_{\alpha}, \lambda_{\alpha}}|^{2^*-1} w_{\alpha} \, \ud  v_{\hat{g}}| \\
&=C2^{2 / n} S|\int_{B_{\delta_{\alpha}}^{+}}(-\Delta_{\hat{g}} U_{\xi_{\alpha}, \lambda_{\alpha}}+O(U_{\xi_{\alpha}, \lambda_{\alpha}}))w_{\alpha} \, \ud  v_{\hat{g}}|\\
&=O(\int_{B_{\delta_{\alpha}}^{+}}U_{\xi_{\alpha}, \lambda_{\alpha}}w_{\alpha} \, \ud  v_{\hat{g}})\\
&\leq C\|U_{\xi_{\alpha}, \lambda_{\alpha}}\|_{L^{2^{*\prime}}(B_{\delta_{\alpha}}^{+})}\|w_{\alpha}\|\\
&\leq C\lambda_{\alpha}^2\|U_{1}\|_{L^{2^{*\prime}}(B_{\lambda_{\alpha}^{-1}}^{+})}\|w_{\alpha}\|.
\end{aligned}
$$
Similarly,
$$
\begin{aligned}
\langle F^{\prime \prime}(0) w_{\alpha}, w_{\alpha}\rangle=& \frac{2}{(\int_{B_{\delta_{\alpha}}^{+}}|t_{\alpha} U_{\xi_{\alpha}, \lambda_{\alpha}}|^{2^*} \, \ud  v_{\hat{g}})^{2 / 2^*}}\Big\{\int_{B_{\delta_{\alpha}}^{+}}|\nabla_{\hat{g}} w_{\alpha}|^{2} \, \ud  v_{\hat{g}}\\
&-(2^*-1) \frac{\int_{B_{\delta_{\alpha}}^{+}}|t_{\alpha} \nabla_{\hat{g}} U_{\xi_{\alpha}, \lambda_{\alpha}}|^{2} \, \ud  v_{\hat{g}}}{\int_{B_{\delta_{\alpha}}^{+}}|t_{\alpha} U_{\xi_{\alpha}, \lambda_{\alpha}}|^{2^*} \, \ud  v_{\hat{g}}} \int_{B_{\delta_{\alpha}}^{+}}|t_{\alpha} U_{\xi_{\alpha}, \lambda_{\alpha}}|^{2^*-2} w_{\alpha}^{2} \, \ud  v_{\hat{g}}\Big\} \\
&+O\Big(\int_{B_{\delta_{\alpha}}^{+}}|U_{\xi_{\alpha}, \lambda_{\alpha}}|^{2^*-1} w_{\alpha} \, \ud  v_{\hat{g}}\Big)^{2}.
\end{aligned}
$$
By Lemma \ref{lem:lemcoercive1}, we have
$$
\int_{B_{\delta_{\alpha}}^{+}}|\nabla_{\hat{g}} w_{\alpha}|^{2} \, \ud  v_{\hat{g}}-(2^*-1) \frac{\int_{B_{\delta_{\alpha}}^{+}}|t_{\alpha} \nabla_{\hat{g}} U_{\xi_{\alpha}, \lambda_{\alpha}}|^{2} \, \ud  v_{\hat{g}}}{\int_{B_{\delta_{\alpha}}^{+}}|t_{\alpha} U_{\xi_{\alpha}, \lambda_{\alpha}}|^{2^*} \, \ud  v_{\hat{g}}} \int_{B_{\delta_{\alpha}}^{+}}|t_{\alpha} U_{\xi_{\alpha}, \lambda_{\alpha}}|^{2^*-2} w_{\alpha}^{2} \, \ud  v_{\hat{g}}\geq\frac{c_{1}}{2}\|w_{\alpha}\|^{2}
$$
for large $\alpha .$ It follows that
$$
\langle F^{\prime \prime}(0) w_{\alpha}, w_{\alpha}\rangle \geq C\|w_{\alpha}\|^{2}+O(\mu_{\alpha}^{4})\|U_{1}\|_{L^{2^{*\prime}}
(B_{\mu_{\alpha}^{-1}}^{+})}^{2}\|w_{\alpha}\|^{2}.
$$
Noticing that $\mu_{\alpha}^2\|U_{1}\|_{L^{2^{*\prime}}(B_{\mu_{\alpha}^{-1}}^{+})} \rightarrow 0$ as $\alpha \rightarrow \infty,$ we have
$$
\begin{aligned}
F(w_{\alpha}) &=F(0)+F^{\prime}(0) w_{\alpha}+\frac{1}{2}\langle F^{\prime \prime}(0) w_{\alpha}, w_{\alpha}\rangle+o(\|w_{\alpha}\|^{2}) \\
& \geq F(0)+O(\mu_{\alpha}^2\|U_{1}\|_{L^{2^{*\prime}}(B_{\mu_{\alpha}^{-1}}^{+})}\|w_{\alpha}\|).
\end{aligned}
$$
By \eqref{eq:Ygualpha}, we conclude that
\be\label{eq:Ygugeq}
Y_g(u_{\alpha}) \geq \frac{\int_{B_{\delta_{\alpha}}^{+}}|\nabla_{\hat{g}} U_{\xi_{\alpha}, \lambda_{\alpha}}|^{2} \, \ud  v_{\hat{g}}}{(\int_{B_{\delta_{\alpha}}^{+}} U_{\xi_{\alpha}, \lambda_{\alpha}}^{2^*} \, \ud  v_{\hat{g}})^{2 / 2^*}}+O(\mu_{\alpha}^2\|U_{1}\|_{L^{2^{*\prime}}(B_{\mu_{\alpha}^{-1}}^{+})}\|w_{\alpha}\|+\mu_{\alpha}^{n-2}).
\ee

\begin{lemma}\label{lem:lem4.1}
We have
$$
\frac{\int_{B_{\delta_{\alpha}}^{+}}|\nabla_{\hat{g}} U_{\xi_{\alpha}, \lambda_{\alpha}}|^{2} \, \ud  v_{\hat{g}}}{(\int_{B_{\delta_{\alpha}}^{+}} U_{\xi_{\alpha}, \lambda_{\alpha}}^{2^*} \, \ud  v_{\hat{g}})^{2 / 2^*}}=\left\{\begin{array}{ll}
\frac{1}{2^{2/n}S}+O(\mu_{\alpha}^{2}), & n \geq 5, \\
\frac{1}{2^{2/n}S}+O(\mu_{\alpha}^{2} \log \mu_{\alpha}^{-1}), & n=4.
\end{array}\right.
$$
\end{lemma}

\begin{proof}
Since $\sqrt{\operatorname{det} \hat{g}_{i j}}=1+O(|z|)$ in $B_{\delta_{\alpha}}^{+},$ we have
$$
\begin{aligned}
\Big(\int_{B_{\delta_{\alpha}}^{+}} U_{\xi_{\alpha}, \lambda_{\alpha}}^{2^*} \, \ud v_{\hat{g}}\Big)^{2 /2^*}=&\Big(\int_{B_{\delta_{\alpha}}^{+}} U_{\xi_{\alpha}, \lambda_{\alpha}}^{2^*}(1+O(|z|)) \, \ud z\Big)^{2 / 2^*} \\
=&\Big(\int_{B_{\delta_{\alpha}}^{+}} U_{\xi_{\alpha}, \lambda_{\alpha}}^{2^*} \, \ud z\Big)^{2 / 2^*}+O\Big(\int_{B_{\delta_{\alpha}}^{+}} U_{\xi_{\alpha}, \lambda_{\alpha}}^{2^*}|z| \, \ud z\Big) \\
=&\Big(\int_{\mathbb{R}_{+}^{n}} U_{\xi_{\alpha}, \lambda_{\alpha}}^{2^*} \, \ud z\Big)^{2 / 2^*}+O(\lambda_{\alpha}^{n-2}) \\
&+O\Big(\lambda_{\alpha}\int_{B_{\lambda_{\alpha}^{-1}}^{+}}(1+c(n)|y|^{2})^{-n}|y| \, \ud y\Big) \\
=& 1+O(\lambda_{\alpha}^{n-2})+O\Big(\lambda_{\alpha}\int_{B_{\lambda_{\alpha}^{-1}}^{+}}(1+c(n)|y|^{2})^{-n}|y| \, \ud y\Big),
\end{aligned}
$$
where we used $\lambda_{\alpha}^{-1}|\xi_{\alpha}| \rightarrow 0$ as $\alpha \rightarrow \infty$.

In addition, by Lemma \ref{lem:Fermi}, we have
$$
\begin{aligned}
\int_{B_{\delta_{\alpha}}^{+}}|\nabla_{\hat{g}} U_{\xi_{\alpha}, \lambda_{\alpha}}|^{2} \, \ud  v_{\hat{g}}=& \int_{B_{\delta_{\alpha}}^{+}}|\nabla U_{\xi_{\alpha}, \lambda_{\alpha}}|^{2} \, \ud  z+2 \hat{h}^{i j}(0) \int_{B_{\delta_{\alpha}}^{+}} \partial_{i} U_{\xi_{\alpha}, \lambda_{\alpha}} \partial_{j} U_{\xi_{\alpha}, \lambda_{\alpha}} z_{n} \, \ud  z \\
&+O\Big(\int_{B_{\delta_{\alpha}}^{+}}|\nabla U_{\xi_{\alpha}, \lambda_{\alpha}}|^{2}|z|^{2} \, \ud  z\Big).
\end{aligned}
$$
It is easy to see that
$$
\int_{B_{\delta_{\alpha}}^{+}}|\nabla U_{\xi_{\alpha}, \lambda_{\alpha}}|^{2} \, \ud  z=\int_{\mathbb{R}_{+}^{n}}|\nabla U_{\xi_{\alpha}, \lambda_{\alpha}}|^{2} \, \ud  z+O(\lambda_{\alpha}^{n-2})=\frac{1}{2^{2/n}S}+O(\lambda_{\alpha}^{n-2})
$$
and
$$
\int_{B_{\delta_{\alpha}}^{+}}|\nabla U_{\xi_{\alpha}, \lambda_{\alpha}}|^{2}|z|^{2} \, \ud  z=C(n) \lambda_{\alpha}^{2} \int_{B_{\lambda_{\alpha}^{-1}}^{+}}(1+c(n)|y|^{2})^{-n}|y|^2|y+\lambda_{\alpha}^{-1} \xi_{\alpha}|^{2} \, \ud  y,
$$
where $C(n)$ is a constant depending only on $n$. By symmetry, we have
$$
\begin{aligned}
&\sum_{i, j=1}^{n-1} \hat{h}^{i j}(0) \int_{B_{\delta_{\alpha}}^{+}} \partial_{i} U_{\xi_{\alpha}, \lambda_{\alpha}} \partial_{j} U_{\xi_{\alpha}, \lambda_{\alpha}} z_{n} \, \ud  z \\
=&\sum_{i, j=1}^{n-1} \hat{h}^{i j}(0) \int_{B_{\delta_{\alpha} / 2}^{+}(\xi_{\alpha})} \partial_{i} U_{\xi_{\alpha}, \lambda_{\alpha}} \partial_{j} U_{\xi_{\alpha}, \lambda_{\alpha}} z_{n} \, \ud  z+O(\mu_{\alpha}^{n-2}) \\
=&\sum_{i=1}^{n-1} \hat{h}^{i i}(0) \int_{B_{\delta_{\alpha} / 2}^{+}(\xi_{\alpha})}|\partial_{1} U_{\xi_{\alpha}, \lambda_{\alpha}}|^{2} z_{n} \, \ud  z+O(\mu_{\alpha}^{n-2})\\
=&O(\mu_{\alpha}^{n-2}),
\end{aligned}
$$
where we used $\sum_{i=1}^{n-1} \hat{h}^{i i}(0)=0$ since $h_{\hat{g}}$ is vanishing at $Q_{\alpha}$. The lemma follows immediately from Lemma \ref{lem:lem3.1}.
\end{proof}

{\bf \noindent Proof of Theorem \ref{thm:mainthm2}.}\quad Notice that
$$
\frac{1}{2^{2/n}S}>I_{\alpha}(u_{\alpha})=Y_g(u_{\alpha})+\alpha\|u_{\alpha}\|_{L^{r}(\partial M)}^{2}.
$$
By \eqref{eq:Ygugeq} and Lemma \ref{lem:lem4.1}, we have
$$
\alpha\|u_{\alpha}\|_{L^{r}(\partial M)}^{2}\leq O(\mu_{\alpha}^{2})+O(\mu_{\alpha}^2\|U_{1}\|_{L^{2^{*\prime}}(B_{\mu_{\alpha}^{-1}}^{+})}\|w_{\alpha}\|+
\mu_{\alpha}^{n-2}).
$$
By Proposition \ref{pro:proEnergyEstimation}, we find
\begin{equation}
\begin{aligned}
\alpha\|u_{\alpha}\|_{L^{r}(\partial M)}^{2} \leq C\Big\{&\mu_{\alpha}^2\|U_{1}\|_{L^{2^{*\prime}}(B_{\mu_{\alpha}^{-1}}^{+})}
\Big(\mu_{\alpha}^2\|U_{1}\|_{L^{2^{*\prime}}(B_{\mu_{\alpha}^{-1}}^{+})}\\
&+\varepsilon_{\alpha}\|U_{1}^{r-1}\|_{L^{r}(\partial^{\prime} B_{\mu_{\alpha}^{-1}}^{+})}+\mu_{\alpha}^{n-2}\|U_{1}^{2^*-2}\|_{L^{2^{*\prime}}(B_{\mu_{\alpha}^{-1}}^{+})}\Big)+
\mu_{\alpha}^{2}\Big\}.
\end{aligned}
\end{equation}
Due to $n \geq 7,$ we have
$$
\begin{aligned}
\|U_{1}\|_{L^{2^{*\prime}}(B_{\mu_{\alpha}^{-1}}^{+})} & \leq C, \\
\mu_{\alpha}^2\|U_{1}^{r-1}\|_{L^{r}(\partial^{\prime} B_{\mu_{\alpha}^{-1}}^{+})} & \leq C \mu_{\alpha}^2(1+\mu_{\alpha}^{\frac{n^{2}-8n+8}{2 n}})=o(1), \\
\|U_{1}^{2^*-2}\|_{L^{2^{*\prime}}(B_{\mu_{\alpha}^{-1}}^{+})} & \leq C\mu_{\alpha}^{(6-n) / 2}.
\end{aligned}
$$
From \eqref{eq:varepsilon2}, i.e., $\varepsilon_{\alpha} \leq C\alpha\|u_{\alpha}\|_{L^{r}(\partial M)}^{2},$ it follows that
$$
\alpha\|u_{\alpha}\|_{L^{r}(\partial M)}^{2} \leq C \mu_{\alpha}^{2}.
$$
On the other hand, rescaling, we have:
$$
\|u_{\alpha}\|_{L^{r}(\partial M)} \geq\|u_{\alpha}\|_{L^{r}(B_{\mu_{\alpha}}^{+}(Q_{\alpha}) \cap \partial M)} \geq C \mu_{\alpha}\|U_1\|_{L^{r}(\partial^{\prime} B_{1}^{+})} \geq C\mu_{\alpha}.
$$
Hence,
$$
\alpha \leq C.
$$
This is a contradiction. Hence, Theorem \ref{thm:mainthm2} is established for all $n\geq7$.

\section{Proof of Theorem \ref{thm:mainthm3}}

The proof of Theorem \ref{thm:mainthm3} is very similar to the proof of Theorem \ref{thm:mainthm2}. We now begin to prove Theorem \ref{thm:mainthm3} by a contradiction argument. Define
$$J_{\alpha}(u)=\frac{\int_{M}|\nabla_{g} u|^{2}\, \ud  v_{g}+\frac{n-2}{2} \int_{\partial M} h_{g} u^{2}\, \ud s_g+\alpha(\|u\|_{L^{r_1}(M)}^{2}+\|u\|_{L^{r_2}(\partial M)}^{2})}{{(\int_{M}|u|^{2^*}\, \ud v_{g} )^{2/2^*}}}$$
for all $u \in H^{1}(M) \backslash \{0\}$. It follows from the contradiction hypothesis that for all large $\alpha,$
\be\label{eq:5-1}
\zeta_{\alpha}:=\inf _{H^{1}(M) \backslash \{0\}} J_{\alpha}<\frac{1}{2^{2 / n} S}.
\ee
As shown in section 2, we know that there exists some nonnegative $u_{\alpha} \in H^{1}(M)$ such that
\be\label{eq:5-2}
\zeta_{\alpha}=\int_{M}|\nabla_{g} u_{\alpha}|^{2} \, \ud  v_{g}+\frac{n-2}{2} \int_{\partial M} h_{g} u_{\alpha}^{2}\, \ud  s_{g}+\alpha(\|u_{\alpha}\|_{L^{r_1}(M)}^{2}+\|u_{\alpha}\|_{L^{r_2}(\partial M)}^{2}),
\ee
\be\label{eq:2*norm=1-}
\int_{M} u_{\alpha}^{2^*} \, \ud  v_{g}=1.
\ee
Therefore, $u_{\alpha}$ satisfies
\be\label{eq:5-3}
\left\{\begin{array}{ll}
-\Delta_{g} u_{\alpha}=\zeta_{\alpha} u_{\alpha}^{2^*-1}-\alpha\|u_{\alpha}\|_{L^{r_1}(M)}^{2-r_1} u_{\alpha}^{r_1-1} & \text { in } M, \\
\frac{\partial_{g} u_{\alpha}}{\partial \nu}=-\frac{n-2}{2} h_{g} u_{\alpha}-\alpha\|u_{\alpha}\|_{L^{r_2}(\partial M)}^{2-r_2} u_{\alpha}^{r_2-1} & \text { on } \partial M.
\end{array}\right.
\ee

\begin{lemma}\label{lem:lem5.1}
For every $\varepsilon>0,$ there exists a constant $D(\varepsilon)>0$ depending on $\varepsilon$, $M$, and $g$ such that
$$
\Big(\int_{M}|u|^{2^*} \, \ud  v_{g}\Big)^{2 / 2^*} \leq(2^{2 / n} S+\varepsilon) \int_{M}|\nabla_{g} u|^{2} \, \ud v_{g}+D(\varepsilon)(\|u\|_{L^{r_1}(M)}^{2}+\|u\|_{L^{r_2}(\partial M)}^{2})
$$
for all $u \in H^{1}(M)$.
\end{lemma}

\begin{proof}
By compactness, we have that for every $\varepsilon>0$ there exists positive constant $\tilde{D}(\varepsilon)$ and $\bar{D}(\varepsilon)$ such that
\be\label{eq:Interpolation Inequality2}
\int_{M} u^{2} \, \ud  v_{g} \leq \varepsilon \int_{M}|\nabla_{g} u|^{2} \, \ud  v_{g}+\tilde{D}(\varepsilon)\|u\|_{L^{r_1}(M)}^{2}
\ee
and
\be\label{eq:Interpolation Inequality3}
\int_{\partial M} u^{2} \, \ud  s_{g} \leq \varepsilon \int_{M}|\nabla_{g} u|^{2} \, \ud  v_{g}+\bar{D}(\varepsilon)\|u\|_{L^{r_2}(\partial M)}^{2}.
\ee
Hence, the lemma follows from the inequality \eqref{eq:LZ98-b}.
\end{proof}

It follows from \eqref{eq:Interpolation Inequality3} that
$$
\begin{aligned}
J_{\alpha}(u_{\alpha}) =&\int_{M}|\nabla_{g} u_{\alpha}|^{2} \, \ud  v_{g}+\frac{n-2}{2} \int_{\partial M} h_{g} u_{\alpha}^{2}\, \ud s_g+\alpha(\|u_{\alpha}\|_{L^{r_1}(M)}^{2}+\|u_{\alpha}\|_{L^{r_2}(\partial M)}^{2}) \\
\geq&(1-\varepsilon \max _{\partial M}|h_{g}|) \int_{M}|\nabla_{g} u_{\alpha}|^{2} \, \ud  v_{g}+\alpha\|u_{\alpha}\|_{L^{r_1}(M)}^{2}+(\alpha-{E}(\varepsilon))\|u_{\alpha}\|_{L^{r_2}(\partial M)}^{2},
\end{aligned}
$$
where ${E}(\varepsilon)$ is a positive constant depending on $\varepsilon$, $M$, and $g$. Then we derive by \eqref{eq:5-1} that
$$
\int_{M}|\nabla_{g} u_{\alpha}|^{2} \, \ud  v_{g} \leq \frac{2}{2^{2/n}S}
$$
and
$$
\|u_{\alpha}\|_{L^{r_1}(M)} \rightarrow 0,\ \|u_{\alpha}\|_{L^{r_2}(\partial M)} \rightarrow 0 \quad \text { as } \alpha \rightarrow \infty.
$$
It follows that $u_{\alpha} \rightharpoonup \bar{u}$ in $H^{1}(M)$ for some $\bar{u} \in H_{0}^{1}(M)$.

We claim that, as $\alpha \rightarrow \infty$,
\be
\zeta_{\alpha} \rightarrow \frac{1}{2^{2/n}S}
\ee
and
\be\label{eq:5-7}
\alpha\|u_{\alpha}\|_{L^{r_1}(M)}^{2} \rightarrow 0,\quad \alpha\|u_{\alpha}\|_{L^{r_2}(\partial M)}^{2} \rightarrow 0.
\ee
Indeed, by Lemma \ref{lem:lem5.1} and \eqref{eq:Interpolation Inequality3}, for every $\varepsilon>0$,
$$
\begin{aligned}
1 & \leq(2^{2/n}S+\varepsilon/2) \int_{M}|\nabla_{g} u_{\alpha}|^{2} \, \ud  v_{g}+D(\varepsilon)(\|u_{\alpha}\|_{L^{r_1}(M)}^{2}+\|u_{\alpha}\|_{L^{r_2}(\partial M)}^{2}) \\
& \leq(2^{2/n}S+\varepsilon) \zeta_{\alpha}+(D(\varepsilon)-\alpha 2^{2/n}S)\|u_{\alpha}\|_{L^{r_1}(M)}^{2}+(2 D(\varepsilon)-\alpha 2^{2/n}S)\|u_{\alpha}\|_{L^{r_2}(\partial M)}^{2}.
\end{aligned}
$$
Thus
$$
\frac{1}{2^{2/n}S+\varepsilon} \leq \zeta_{\alpha}<\frac{1}{2^{2/n}S}
$$
and
$$\frac{1}{2} \alpha 2^{2/n}S(\|u_{\alpha}\|_{L^{r_1}(M)}^{2}+\|u_{\alpha}\|_{L^{r_2}(\partial M)}^{2}) \leq(2^{2/n}S+\varepsilon) \frac{1}{2^{2/n}S}-1=\frac{\varepsilon}{2^{2/n}S},
$$
if $\alpha 2^{2/n}S>4 D(\varepsilon) .$ Hence, the claim follows.

Let $x_{\alpha}$ be some maximum point of $u_{\alpha},$ set $\mu_{\alpha}:=u_{\alpha}(x_{\alpha})^{-2 /(n-2)}$.

\begin{lemma}\label{lem:lem5.2}
We have
$$
\lim _{\alpha \rightarrow \infty} \alpha \mu_{\alpha}^{2}=0.
$$
\end{lemma}

\begin{proof}
The proof is similar to that of Lemma \ref{lem:lemalphato0}. We first claim
\be\label{eq:5-8}
\liminf _{\alpha \rightarrow \infty}\|u_{\alpha}\|_{L^{2(n-1)/(n-2)}(\partial M)}>0.
\ee
If the claim were false, i.e., $\|u_{\alpha}\|_{L^{2(n-1)/(n-2)}(\partial M)} \rightarrow 0$ along a subsequence $\alpha \rightarrow \infty$. Then, in view of \eqref{eq:5-1} and \eqref{eq:5-2}, there would exist $\hat{u} \in H_{0}^{1}(M)$ such that $u_{\alpha}$ weakly converges to $\hat{u}$.  It follows from Br\'ezis-Lieb lemma that $u_{\alpha}$ and $\hat{u}$ satisfy
\be\label{eq:BrezisLieb1-}
\int_{M} u_{\alpha}^{2^*}\, \ud v_g-\int_{M}|u_{\alpha}-\hat{u}|^{2^*}\, \ud v_g-\int_{M} \hat{u}^{2^*}\, \ud v_g \rightarrow 0 \quad \text { as } \alpha \rightarrow \infty,
\ee
and, in view of \eqref{eq:2*norm=1-},
\be\label{eq:BrezisLieb2-}
\int_{M}|u_{\alpha}-\hat{u}|^{2^*}\, \ud v_g \leq 1+o(1), \quad \int_{M} \hat{u}^{2^*}\, \ud v_g \leq 1.
\ee
By Lemma \ref{lem:pro1.3} applied to $u_{\alpha}-\hat{u},$ we have that for any $\varepsilon>0$,
\be\label{eq:lemmaapplied-}
\int_{M}|\nabla_{g}(u_{\alpha}-\hat{u})|^{2}\, \ud v_g\geq(1 / S-\varepsilon)\Big(\int_{M}|u_{\alpha}-\hat{u}|^{2^*}\, \ud v_g\Big)^{2 / 2^*}.
\ee
By the definition of $\zeta_{\alpha},$ we have
$$
\int_{M}|\nabla_{g} \hat{u}|^{2}\, \ud v_g\geq \zeta_{\alpha}\Big(\int_{M} \hat{u}^{2^*}\, \ud v_g\Big)^{2/2^*}-\alpha\|\hat{u}\|_{L^{r_1}(M)}^2.
$$
From \eqref{eq:5-7} we know $\|\hat{u}\|_{L^{r_1}(M)}=0$, then,
\be\label{eq:defiell-}
\int_{M}|\nabla_{g} \hat{u}|^{2}\, \ud v_g\geq \zeta_{\alpha}\Big(\int_{M} \hat{u}^{2^*}\, \ud v_g\Big)^{2/2^*}.
\ee
By the Sobolev embedding theorem, \eqref{eq:BrezisLieb1-}, \eqref{eq:BrezisLieb2-}, \eqref{eq:lemmaapplied-} and \eqref{eq:defiell-}, we have, for $\varepsilon>0$,
$$
\begin{aligned}
\zeta_{\alpha} =&\int_{M}|\nabla_{g} u_{\alpha}|^{2}\, \ud v_g +\frac{n-2}{2} \int_{\partial M} h_{g} u_{\alpha}^{2}\, \ud s_g+\alpha(\|u_{\alpha}\|_{L^{r_1}(M)}^{2}+\|u_{\alpha}\|_{L^{r_2}(\partial M)}^{2}) \\
=&\int_{M}|\nabla_{g}(u_{\alpha}-\hat{u})|^{2}\, \ud v_g+\int_{M}|\nabla_{g} \hat{u}|^{2}\, \ud v_g+\alpha(\|u_{\alpha}\|_{L^{r_1}(M)}^{2}+\|u_{\alpha}\|_{L^{r_2}(\partial M)}^{2})+o(1) \\
\geq&(1 / S-\varepsilon)\Big(\int_{M}|u_{\alpha}-\hat{u}|^{2^*}\, \ud v_g\Big)^{2 / 2^*}+\zeta_{\alpha}\Big(\int_{M} \hat{u}^{2^*}\, \ud v_g\Big)^{2 / 2^*}\\
&+\alpha(\|u_{\alpha}\|_{L^{r_1}(M)}^{2}+\|u_{\alpha}\|_{L^{r_2}(\partial M)}^{2})+o(1) \\
\geq&(1 / S-\varepsilon)\int_{M}|u_{\alpha}-\hat{u}|^{2^*}\, \ud v_g+\zeta_{\alpha}\int_{M} \hat{u}^{2^*}\, \ud v_g+\alpha(\|u_{\alpha}\|_{L^{r_1}(M)}^{2}+\|u_{\alpha}\|_{L^{r_2}(\partial M)}^{2})+o(1) \\
=&(1 / S-\varepsilon-\zeta_{\alpha}) \int_{M}|u_{\alpha}-\hat{u}|^{2^*}\, \ud v_g+\zeta_{\alpha}+\alpha(\|u_{\alpha}\|_{L^{r_1}(M)}^{2}+\|u_{\alpha}\|_{L^{r_2}(\partial M)}^{2})+o(1).
\end{aligned}
$$
Taking $\varepsilon$ small enough, we derive by \eqref{eq:5-1} that $\|u_{\alpha}-\hat{u}\|_{L^{2^*}(M)} \rightarrow 0 .$ In particular, in view of \eqref{eq:2*norm=1-}, $\int_{M} \hat{u}^{2^*}\, \ud v_g=1 .$ This contradicts to $\|\hat{u}\|_{L^{r_1}(M)}=0$.

It follows from \eqref{eq:5-8}, the definition of $\mu_{\alpha},$ and \eqref{eq:5-7} that as $\alpha \rightarrow \infty$,
\be\label{eq:varepsilon-1}
\alpha \mu_{\alpha}^{2} \leq C \alpha \mu_{\alpha}^{2} \Big(\int_{\partial M} u_{\alpha}^{\frac{2(n-1)}{n-2}} \, \ud  s_{g}\Big)^{2/r_2} \leq C \alpha \Big(\int_{\partial M} u_{\alpha}^{r_2}\, \ud s_g\Big)^{2/r_2} \rightarrow 0.
\ee
The proof of Lemma \ref{lem:lem5.2} is completed.
\end{proof}

Similar to Proposition \ref{pro:proMoser}, we have
\begin{proposition}\label{pro:proMoser2}
There exists some constant $C$ independent of $\alpha$ such that for all $\alpha \geq 1$,
\be\label{eq:Moser2}
u_{\alpha} / \varphi_{\alpha} \leq C\quad \text { for }\quad x \in \overline{M},
\ee
where $\varphi_{\alpha}$ was defined in \eqref{eq:varphialpha}.
\end{proposition}

Let $Q_{\alpha} \in \partial M$ be the closest point to $x_{\alpha}$. For some small $\delta_{0}$, let $\psi_{\alpha} \in C^{\infty}(\overline{M})$ satisfy $\psi_{\alpha}(Q_{\alpha})=1$, $1/2 \leq \psi_{\alpha} \leq 2$, $\|\psi_{\alpha}\|_{C^{2}(\overline{M})} \leq C,$ and
$$
\left\{\begin{array}{ll}
\Delta_{g} \psi_{\alpha}=0 & \text { in } B_{2 \delta_{0}}^{+}, \\
\frac{\partial_{g} \psi_{\alpha}}{\partial \nu}+\frac{n-2}{2} h_{g} \psi_{\alpha}=0 & \text { on } \partial^{\prime} B_{2 \delta_{0}}^{+}.
\end{array}\right.
$$
Here we used the Fermi coordinate with respect to metric $g$ centered at $Q_{\alpha}$. Set $\hat{g}=\psi_{\alpha}^{4 /(n-2)} g .$ {It is easy to see that $h_{\hat{g}}=0$ on $\partial^{\prime} B_{\delta_{0}}^{+} .$} Hence, $u_{\alpha} / \psi_{\alpha}$ satisfies
\be\label{eq:u/psi2}
\left\{\begin{array}{ll}
-\Delta_{\hat{g}} \frac{u_{\alpha}}{\psi_{\alpha}}=\zeta_{\alpha} (\frac{u_{\alpha}}{\psi_{\alpha}})^{2^*-1}-\alpha\|u_{\alpha}\|_{L^{r_1}(M)}^{2-r_1}\psi_{\alpha}^{-\frac{n+2}{n-2}} u_{\alpha}^{r_1-1} & \text { in } B_{2 \delta_{0}}^{+}, \\
\frac{\partial_{\hat{g}}}{\partial \nu} \frac{u_{\alpha}}{\psi_{\alpha}}+\alpha\|u_{\alpha}\|_{L^{r_2}(\partial M)}^{2-r_2} \psi_{\alpha}^{-\frac{n}{n-2}} u_{\alpha}^{r_2-1}=0 & \text { on } \partial^{\prime} B_{2 \delta_{0}}^{+}.
\end{array}\right.
\ee

Set
$$
w_{\alpha}=\frac{u_{\alpha}}{\psi_{\alpha}}-\chi_{\alpha}-t_{\alpha} \sigma_{\xi_{\alpha}, \lambda_{\alpha}},
$$
where $\chi_{\alpha}$, $t_{\alpha}$, and $\sigma_{\xi_{\alpha}, \lambda_{\alpha}}$ are defined in section 3.

In order to estimate $w_{\alpha},$ we begin by writing an equation for $w_{\alpha}$:

\begin{lemma}\label{lem:lemwalphaEquation-}
$w_{\alpha}$ satisfies
\be\label{eq:walphaEquation-}
\left\{\begin{array}{ll}
-\Delta_{\hat{g}} w_{\alpha}=k_{\alpha}|\Theta_{\alpha}|^{2^*-3} \Theta_{\alpha} w_{\alpha}+b^{\prime}|\Theta_{\alpha}|^{2^*-3} w_{\alpha}^{2}+b^{\prime \prime}|w_{\alpha}|^{2^*-1}+f_{\alpha} & \text { in } B_{\delta_{\alpha}}^{+}, \\
\frac{\partial_{\hat{g}}w_{\alpha}}{\partial \nu}=-\alpha\|u_{\alpha}\|_{L^{r_2}(\partial M)}^{2-r_2} \psi_{\alpha}^{-\frac{n}{n-2}} u_{\alpha}^{r_2-1} & \text { on } \partial^{\prime} B_{\delta_{\alpha}}^{+},
\end{array}\right.
\ee
where
$$
\begin{aligned}
k_{\alpha}&=(2^*-1)\zeta_{\alpha},\\
\Theta_{\alpha}&=t_{\alpha} \sigma_{\alpha}+\chi_{\alpha}, \\
f_{\alpha}&=\zeta_{\alpha} (t_{\alpha}U_{\xi_{\alpha}, \lambda_{\alpha}})^{2^*-1}+t_{\alpha} \Delta_{\hat{g}} U_{\xi_{\alpha}, \lambda_{\alpha}}+O(\mu_{\alpha}^{(n-2) / 2}) U_{\xi_{\alpha}, \lambda_{\alpha}}^{2^*-2}{-\alpha\|u_{\alpha}\|_{L^{r_1}(M)}^{2-r_1}\psi_{\alpha}^{-\frac{n+2}{n-2}} u_{\alpha}^{r_1-1}},
\end{aligned}
$$
and $b^{\prime}$, $b^{\prime \prime}$ are bounded functions with $b^{\prime} \equiv 0$ if $n \geq 6$.
\end{lemma}

\begin{proof}
First of all, by the definition of $w_{\alpha}$, \eqref{eq:u/psi2}, \eqref{eq:equ2} and \eqref{eq:equ3},  we have
\be\label{eq:Deltawalpha-}
\begin{aligned}
-\Delta_{\hat{g}} w_{\alpha}&=-\Delta_{\hat{g}}\Big(\frac{u_{\alpha}}{\psi_{\alpha}}-\chi_{\alpha}-t_{\alpha} \sigma_{\alpha}\Big)\\
&=\zeta_{\alpha}(\Theta_{\alpha}+w_{\alpha})^{2^*-1}+t_{\alpha} \Delta_{\hat{g}} \sigma_{\alpha}{-\alpha\|u_{\alpha}\|_{L^{r_1}(M)}^{2-r_1}\psi_{\alpha}^{-\frac{n+2}{n-2}} u_{\alpha}^{r_1-1}}\\
&=\zeta_{\alpha}(\Theta_{\alpha}+w_{\alpha})^{2^*-1}+t_{\alpha} \Delta_{\hat{g}} U_{\xi_{\alpha}, \lambda_{\alpha}}{-\alpha\|u_{\alpha}\|_{L^{r_1}(M)}^{2-r_1}\psi_{\alpha}^{-\frac{n+2}{n-2}} u_{\alpha}^{r_1-1}}
\end{aligned}\quad \text { in } B_{\delta_{\alpha}}^{+},
\ee
and
$$
\begin{aligned}
\frac{\partial_{\hat{g}} w_{\alpha}}{\partial \nu} &=-\alpha\|u_{\alpha}\|_{L^{r}(\partial M)}^{2-r} \psi_{\alpha}^{-\frac{n}{n-2}} u_{\alpha}^{r-1}-t_{\alpha} \frac{\partial_{\hat{g}} \sigma_{\alpha}}{\partial \nu} \\
&=-\alpha\|u_{\alpha}\|_{L^{r}(\partial M)}^{2-r} \psi_{\alpha}^{-\frac{n}{n-2}} u_{\alpha}^{r-1}-t_{\alpha} \frac{\partial_{\hat{g}} U_{\xi_{\alpha}, \lambda_{\alpha}}}{\partial \nu} \\
&=-\alpha\|u_{\alpha}\|_{L^{r}(\partial M)}^{2-r} \psi_{\alpha}^{-\frac{n}{n-2}} u_{\alpha}^{r-1}
\end{aligned}\quad \text { on } \partial^{\prime} B_{\delta_{\alpha}}^{+},
$$
where we used that $\frac{\partial_{\hat{g}} U_{\xi_{\alpha}, \lambda_{\alpha}}}{\partial \nu}=\frac{\partial U_{\xi_{\alpha}, \lambda_{\alpha}}}{\partial x_{n}}=0$ in Fermi coordinate systems. By the proof of Lemma \ref{lem:lemwalphaEquation}, we have 
$$
(\Theta_{\alpha}+w_{\alpha})^{2^*-1}=|\Theta_{\alpha}|^{2^*-2} \Theta_{\alpha}+(2^*-1)|\Theta_{\alpha}|^{2^*-3} \Theta_{\alpha} w_{\alpha} +b^{\prime}|\Theta_{\alpha}|^{2^*-3} w_{\alpha}^{2}+b^{\prime \prime}|w_{\alpha}|^{2^*-1}
$$
and
$$
|\Theta_{\alpha}|^{2^*-2} \Theta_{\alpha}=(t_{\alpha} U_{\xi_{\alpha}, \lambda_{\alpha}})^{2^*-1}+O(\mu_{\alpha}^{(n-2) / 2} U_{\xi_{\alpha}, \lambda_{\alpha}}^{2^*-2}),
$$
where $b^{\prime}, b^{\prime \prime}$ are bounded functions and $b^{\prime} \equiv 0$ if $n \geq 6$. Inserting the above expansions into \eqref{eq:Deltawalpha-}, we obtain \eqref{eq:walphaEquation-}.
\end{proof}

\begin{proposition}\label{pro:proEnergyEstimation2}
We have
$$
\begin{aligned}
\|w_{\alpha}\|\leq C\Big\{&\mu_{\alpha}^2\|U_{1}\|_{L^{2^{*\prime}}(B_{\mu_{\alpha}^{-1}}^{+})}+\varepsilon_{\alpha}\|U_{1}^{r_2-1}\|_{L^{r_2}(\partial^{\prime} B_{\mu_{\alpha}^{-1}}^{+})}+\tilde{\varepsilon}_{\alpha}\|U_{1}^{r_1-1}\|_{L^{r_1}(B_{\mu_{\alpha}^{-1}}^{+})}\\
&+\mu_{\alpha}^{n-2}\|U_{1}^{2^*-2}\|_{L^{2^{*\prime}}(B_{\mu_{\alpha}^{-1}}^{+})}\Big\},
\end{aligned}
$$
where $\varepsilon_{\alpha}$ is given in \eqref{eq:varepsilonalpha}, $\tilde{\varepsilon}_{\alpha}=\alpha \mu_{\alpha}^{n-\frac{n-2}{2} r_1}\|u_{\alpha}\|_{L^{r_1}(M)}^{2-r_1}$.
\end{proposition}

\begin{proof}
Multiplying \eqref{eq:walphaEquation-} by $w_{\alpha}$ and integrating over $B_{\delta_{\alpha}}^+$ we obtain
$$
Q_{\alpha}(w_{\alpha}, w_{\alpha})+o(\|w_{\alpha}\|^{2})=\int_{B_{\delta_{\alpha}}^{+}} w_{\alpha} f_{\alpha} \, \ud  v_{\hat{g}}-\int_{\partial^{\prime} B_{\delta_{\alpha}}^{+}} \alpha\|u_{\alpha}\|_{L^{r}(\partial M)}^{2-r} \psi_{\alpha}^{-\frac{n}{n-2}} u_{\alpha}^{r-1}w_{\alpha} \, \ud  s_{\hat{g}}.
$$
By H$\ddot{\text{o}}$lder's inequality, Sobolev's inequality, and Lemma \ref{lem:lemcoercive}, we have
$$
\begin{aligned}
\|w_{\alpha}\| \leq C\Big\{&\|2^{2 / n} S\Delta_{\hat{g}} U_{\xi_{\alpha}, \lambda_{\alpha}}+U_{\xi_{\alpha}, \lambda_{\alpha}}^{2^*-1}\|_{L^{2^{*\prime}}(B_{\delta_{\alpha}}^{+})}+\alpha\|u_{\alpha}\|_{L^{r_2}(\partial M)}^{2-r_2}\|u_{\alpha}^{r_2-1}\|_{L^{r_2}(\partial^{\prime} B_{\delta_{\alpha}}^{+})}\\
&{+\alpha\|u_{\alpha}\|_{L^{r_1}( M)}^{2-r_1}\|u_{\alpha}^{r_1-1}\|_{L^{2^{*\prime}}(B_{\delta_{\alpha}}^{+})}}+\mu_{\alpha}^{\frac{n-2}{2}}\|U_{\xi_{\alpha}, \lambda_{\alpha}}^{2^*-2}\|_{L^{2^{*\prime}}(B_{\delta_{\alpha}}^{+})}\Big\},
\end{aligned}
$$
where $2^{*\prime}=2n/(n+2)$. By the proof of Proposition \ref{pro:proEnergyEstimation}, we have
$$
\|2^{2 / n} S\Delta_{\hat{g}} U_{\xi_{\alpha}, \lambda_{\alpha}}+U_{\xi_{\alpha}, \lambda_{\alpha}}^{2^*-1}\|_{L^{2^{*\prime}}(B_{\delta_{\alpha}}^{+})}\leq C \lambda_{\alpha}^2\|U_{1}\|_{L^{2^{*\prime}}(B_{\lambda_{\alpha}^{-1}}^{+})},
$$
$$
\|u_{\alpha}^{r_2-1}\|_{L^{r_2}(\partial^{\prime} B_{\delta_{\alpha}}^{+})} \leq C\mu_{\alpha}^{n-1-\frac{n-2}{2} r_2}\|U_{1}^{r_2-1}\|_{L^{r_2}(\partial^{\prime} B_{\mu_{\alpha}^{-1}}^{+})},
$$
and
$$
\|U_{\xi_{\alpha}, \lambda_{\alpha}}^{2^*-2}\|_{L^{2^{*\prime}}(B_{\delta_{\alpha}}^{+})} \leq C \mu_{\alpha}^{\frac{n-2}{2}}\|U_{1}^{2^*-2}\|_{L^{2^{*\prime}}(B_{\mu_{\alpha}^{-1}}^{+})}.
$$
By Corollary \ref{cor:Pointwise Estimation}, we have $u_{\alpha} \leq C U_{\mu_{\alpha}}$. Hence
$$
{\|u_{\alpha}^{r_1-1}\|_{L^{r_1}(B_{\delta_{\alpha}}^{+})} \leq C\|U_{\mu_{\alpha}}^{r_1-1}\|_{L^{r_1}(B_{\delta_{\alpha}}^{+})} \leq C\mu_{\alpha}^{n-\frac{n-2}{2} r_1}\|U_{1}^{r_1-1}\|_{L^{r_1}(B_{\mu_{\alpha}^{-1}}^{+})},}
$$
Therefore, we obtained the estimate for $\|w_{\alpha}\|$.
\end{proof}

{\bf \noindent Proof of Theorem \ref{thm:mainthm3}.}\quad As in section 4, we have
$$
\frac{1}{2^{2/n}S}>J_{\alpha}(u_{\alpha})=K_g(u_{\alpha})+\alpha\|u_{\alpha}\|_{L^{r_1}(M)}^{2}+\alpha\|u_{\alpha}\|_{L^{r_2}(\partial M)}^{2},
$$
where
$$
K_g(u_{\alpha})=\frac{\int_{B_{\delta_{\alpha}}^{+}}|\nabla_{\hat{g}}(\frac{u_{\alpha}}{\psi_{\alpha}})|^{2} \, \ud  v_{\hat{g}}}{(\int_{B_{\delta_{\alpha}}^{+}}(\frac{u_{\alpha}}{\psi_{\alpha}})^{2^*} \, \ud  v_{\hat{g}})^{2 / 2^*}}+O(\mu_{\alpha}^{n-2}).
$$
Therefore we have
$$
\alpha\|u_{\alpha}\|_{L^{r_1}(M)}^{2}+\alpha\|u_{\alpha}\|_{L^{r_2}(\partial M)}^{2}\leq O(\mu_{\alpha}^{2})+O(\mu_{\alpha}^2\|U_{1}\|_{L^{2^{*\prime}}(B_{\mu_{\alpha}^{-1}}^{+})}\|w_{\alpha}\|+\mu_{\alpha}^{n-2}).
$$
By Proposition \ref{pro:proEnergyEstimation2}, we find
\begin{equation}
\begin{aligned}
&\alpha\|u_{\alpha}\|_{L^{r_1}(M)}^{2}+\alpha\|u_{\alpha}\|_{L^{r_2}(\partial M)}^{2}\\
\leq& C\Big\{\mu_{\alpha}^2\|U_{1}\|_{L^{2^{*\prime}}(B_{\mu_{\alpha}^{-1}}^{+})}
\Big(\mu_{\alpha}^2\|U_{1}\|_{L^{2^{*\prime}}(B_{\mu_{\alpha}^{-1}}^{+})}+\varepsilon_{\alpha}\|U_{1}^{r_2-1}\|_{L^{r_2}(\partial^{\prime} B_{\mu_{\alpha}^{-1}}^{+})}\\
&\quad+\tilde{\varepsilon}_{\alpha}\|U_{1}^{r_1-1}\|_{L^{r_1}(B_{\mu_{\alpha}^{-1}}^{+})}+\mu_{\alpha}^{n-2}\|U_{1}^{2^*-2}\|_{L^{2^{*\prime}}(B_{\mu_{\alpha}^{-1}}^{+})}\Big)+\mu_{\alpha}^{2}\Big\}.
\end{aligned}
\end{equation}
Due to $n \geq 7,$ we have
$$
\begin{aligned}
\|U_{1}\|_{L^{2^{*\prime}}(B_{\mu_{\alpha}^{-1}}^{+})} & \leq C, \\
\mu_{\alpha}^2\|U_{1}^{r_2-1}\|_{L^{r_2}(\partial^{\prime} B_{\mu_{\alpha}^{-1}}^{+})} & \leq C \mu_{\alpha}^2=o(1), \\
\mu_{\alpha}^2\|U_{1}^{r_1-1}\|_{L^{r_1}(B_{\mu_{\alpha}^{-1}}^{+})} & \leq C \mu_{\alpha}^2(1+\mu_{\alpha}^{\frac{n^{2}-12n+4}{2(n+2)}})=o(1), \\
\|U_{1}^{2^*-2}\|_{L^{2^{*\prime}}(B_{\mu_{\alpha}^{-1}}^{+})} & \leq C\mu_{\alpha}^{(6-n) / 2}.
\end{aligned}
$$
Moreover, we derive by \eqref{eq:varepsilon2} that $\varepsilon_{\alpha} \leq C\alpha\|u_{\alpha}\|_{L^{r_2}(\partial M)}^{2}$. Similarly, $\tilde{\varepsilon}_{\alpha} \leq C\alpha\|u_{\alpha}\|_{L^{r_1}(M)}^{2}$. It follows that
$$
\alpha\|u_{\alpha}\|_{L^{r_1}(M)}^{2}+\alpha\|u_{\alpha}\|_{L^{r_2}(\partial M)}^{2} \leq C \mu_{\alpha}^{2}.
$$
On the other hand, rescaling, we have:
$$
\|u_{\alpha}\|_{L^{r_1}(M)} \geq\|u_{\alpha}\|_{L^{r_1}(B_{\mu_{\alpha}}^{+}(Q_{\alpha}) \cap M)} \geq C \mu_{\alpha}^2\|U_{1}\|_{L^{r_1}(B_{1}^{+})} \geq C\mu_{\alpha}^2,
$$
$$
\|u_{\alpha}\|_{L^{r_2}(\partial M)} \geq\|u_{\alpha}\|_{L^{r_2}(B_{\mu_{\alpha}}^{+}(Q_{\alpha}) \cap \partial M)} \geq C \mu_{\alpha}\|U_{1}\|_{L^{r_2}(\partial^{\prime} B_{1}^{+})} \geq C\mu_{\alpha}.
$$
Hence,
$$
\alpha \leq C.
$$
This is a contradiction. Hence, Theorem \ref{thm:mainthm3} is established for all $n\geq7$.

\bigskip

\noindent Z. Tang, J. Xiong \& N. Zhou

\medskip

\noindent School of Mathematical Sciences, Laboratory of Mathematics and Complex Systems, MOE, \\
Beijing Normal University, Beijing 100875, China\\[1mm]
Email: \textsf{tangzw@bnu.edu.cn}\\
Email: \textsf{jx@bnu.edu.cn}\\
Email: \textsf{nzhou@mail.bnu.edu.cn}

\end{document}